\documentclass[a4paper,12pt]{amsart}
\usepackage{mathrsfs}
\usepackage{latexsym}
\usepackage{amssymb}
\usepackage{amsmath}
\usepackage{amsfonts}
\usepackage{amsthm}
\usepackage{multirow}
\usepackage{xcolor}
\usepackage{verbatim}
\usepackage{caption}
\usepackage{graphicx}
\usepackage{hyperref}
\usepackage{enumerate}
\usepackage{booktabs}
\usepackage{array}

\theoremstyle{definition}
\newtheorem{definition}{Definition}[section]
\newtheorem{notation}[definition]{Notation}

\theoremstyle{plain}
\newtheorem{theorem}[definition]{Theorem}
\newtheorem{proposition}[definition]{Proposition}
\newtheorem{corollary}[definition]{Corollary}
\newtheorem{lemma}[definition]{Lemma}

\theoremstyle{remark}
\newtheorem{remark}[definition]{Remark}
\newtheorem{example}[definition]{Example}

\def\a{\alpha} \def\b{\beta}   
   
\def\Ga{\Gamma}

\def\ZZ{\mathbb{Z}}    

\def\le{\leqslant} \def\leq{\leqslant} \def\ge{\geqslant} \def\geq{\geqslant}
 
\def\l{\langle} \def\r{\rangle}

\def\calM{\mathcal{M}}         \def\calN{\mathcal{N}}            

\def\Aut{{\mathrm{Aut}}} \def\Inn{{\mathrm{Inn}}}

\def\C{\mathrm{C}}    
\def\D{\mathrm{D}}    
\def\N{\mathrm{N}}    
\def\S{\mathrm{S}}    
\def\RM{\mathsf{RotaMap}}
\def\BiRM{\mathsf{BiRotaMap}}

\title{Embedding a Praeger-Xu graph into a surface}
\date{}
\thanks{This work was supported  by NSFC grant 12350710787}
\author[Ding]{Zhaochen Ding}
	\address{Department of Mathematics\\
 University of Auckland\\
 Auckland Central, Auckland 1010\\
 New Zealand}
	\email{dzha470@aucklanduni.ac.nz}

\author[Guo]{Zheng Guo}
	\address{Department of Mathematics\\
		Southern University of Science and Technology\\
		Shenzhen, Guangdong 518055\\
		P. R. China}
	\email{12131227@mail.sustech.edu.cn}

 \author[Liu]{Luyi Liu}
	\address{Interenational Center for Mathematics\\
		Southern University of Science and Technology\\
		Shenzhen, Guangdong 518055\\
		P. R. China}
	\email{12031108@mail.sustech.edu.cn}
\begin{document}
	
	\begin{abstract}
 Rotary maps (orientably regular maps) are highly symmetric graph embeddings on orientable surfaces. This paper classifies all rotary maps whose underlying graphs are Praeger-Xu graphs, denoted $\operatorname{C}(p,r,s)$, 
 for any odd prime $p$ that does not divide $r$. 
 Our main result establishes a one-to-one correspondence between the isomorphism classes of these maps and the multiplicity-free representations of the dihedral group $\D_{2r}$
  over the finite field $\mathbb{F}_p$. This work extends a recent classification for the case where $p=2$.

	\end{abstract}
	\maketitle
	\section{Introduction}
	A \emph{rotary map} (or \emph{orientably regular map}) is an orientable map whose orientation-preserving automorphism group acts transitively on the arcs of its underlying graph. 
	Such maps provide rich examples of highly symmetric structures on surfaces, revealing deep interconnections between group theory, Galois theory, and the theory of Riemann surfaces, cf. \cite{lando2004Graphs}. 
	A central problem in topological graph theory involves classifying all rotary maps with a specified family of underlying graphs. 
	Early work by James and Jones \cite{james1985Regular} established a complete classification of rotary maps with complete underlying graphs.
	Following sustained collaborative efforts, the classification of rotary maps with complete bipartite underlying graphs has also been fully established \cite{du2007Regular,du2010Regular,jones2010Regular,jones2007Regular}.

	A foundational infinite family of arc-transitive graphs of valency $2p$ (for 
	 prime $p$) arises from the Praeger–Xu graphs $\operatorname{C}(p,r,s)$,  introduced in \cite{praeger1989Characterization}. 
	Despite being historically understudied, Praeger–Xu graphs have proven valuable in advancing the classification and structural analysis of arc-transitive graphs, in works such as \cite{giudici2020Arctransitive,li2024Covers,potocnik2015Bounding,praeger1989Highly}.
	Recent work by Jajcay, Poto\v{c}nik, and Wilson \cite{jajcayPRAEGERXU,jajcay2022Cayleyness} has focused on $4$-valent Praeger–Xu graphs. 
	They characterized their cycle structures, orientably regular embeddings, semitransitive orientations, and vertex-transitive automorphism groups. 
	Their results further unveiled connections to linear codes, and they used the consistent cycle structure of these graphs to construct all rotary maps with underlying graph $\operatorname{C}(2,r,s)$.
	
	Therefore, a question arises: can we classify all rotary embeddings of a Praeger-Xu graph $\operatorname{C}(p,r,s)$?
	
	Although the method of Jajcay, Poto\v{c}nik, and Wilson applies to cases where $p\ge 3$, it faces significant challenges in determining isomorphism between two rotary maps constructed through their approach. 
	To overcome this limitation, we alternatively employ the theory of coset maps introduced by Li, Praeger, and Song \cite{LiPraegerS-A}.

	A \emph{rotary pair} $(\rho,\tau)\in G\times G$ of a group $G$ satisfies $G=\langle\rho,\tau\rangle$ and $|\tau|=2$.
	Such pairs give rise to rotary maps through the coset map construction, and conversely, every rotary map arises from a rotary pair, see \cite[Section 4]{LiPraegerS-A}. 
	This correspondence facilitates the study of rotary maps via rotary pairs. 
	
	As the case where $p=2$ has been well-studied by Jajcay, Poto\v{c}nik, and Wilson, we only consider cases where $p\ge 3$.
	Additionally, we assume $p\nmid r$ in this paper.
	For the case where $p\mid r$, the classification is significantly more difficult.
	We give Example \ref{example:p|r} to illustrate this.
 
	Let $a,b$ be two involutions which generate $\operatorname{D}_{2r}$, and let $\gamma_{-1,1}:\operatorname{D}_{2r}\rightarrow \mathbb{F}_p^\times$
	be a representation of degree $1$ given by
	\[\gamma_{-1,1}(a)v=-v,\gamma_{-1,1}(b)v=v\] for every $v\in\mathbb{F}_p$.
	Suppose $\psi$ is a representation of a group $G$, with irreducible decomposition:
	\[\psi=\sum_{i=1}^m m_i\psi_i.\]
	Then $\psi$ is \emph{multiplicity-free} if $m_i=1$ for every $i\in\{1,\ldots,m\}$.
	Our main results establish that rotary maps with Praeger-Xu underlying graphs are fundamentally tied to multiplicity-free representations of dihedral groups, as follows:
	
	\begin{theorem}\label{thm:main}
		%
		%
		There is a one-to-one correspondence between rotary maps with underlying graph $\operatorname{C}(p,r,s)$ and 
		\begin{enumerate}[{\rm (i)}]
			\item multiplicity-free representations $\psi$ of $\operatorname{D}_{2r}$ over $\mathbb{F}_p$ such that $\deg(\psi) = s+1$ if $r$ is odd;\vskip0.1in
   
			\item multiplicity-free representations $\psi$ of $\operatorname{D}_{2r}$ over $\mathbb{F}_p$ such that $\deg(\psi) = s+1$ and $\gamma_{-1,1}$ is not a constituent of $\psi$ if $r$ is even.
		\end{enumerate}
	\end{theorem}
	
	An important corollary to  Theorem \ref{thm:main} is as follows.
	\begin{corollary}\label{cor:faithful}
		Let $s$ be a positive integer and let $r\ge \max\{3,s+1\}$ be a prime such that $p\ne r$. Let $d$ be the least positive integer such that $r\mid p^d-1$.
        Let \[\zeta=\left\{\begin{aligned}
            d&,\ \mbox{$d$ is even and $p^{d/2}\equiv -1\pmod{r}$},\\
            2d&,\ \mbox{otherwise}.
        \end{aligned}\right.\]
        Then there is a rotary map with underlying graph isomorphic to $\operatorname{C}(p,r,s)$ if and only if $s\equiv -1,0,1\pmod{\zeta}$.
		
	\end{corollary}

	\section{Preliminaries}
	In this section, we introduce some basic definitions and notations for Praeger-Xu graphs, rotary maps and some groups.
	Some known results are presented without proofs.
	In this paper, a graph may have multiple edges but does not contain loops.

	\subsection{Praeger-Xu graphs}
	
	Motivated by questions in non-locally primitive arc-transitive graphs, Praeger and Xu studied vertex valency $4$ arc-transitive graphs -- the smallest vertex valency without local primitivity guarantees. 
	They discovered a family of graphs of valency $2p$, where $p$ is a prime. 
	Such a class of arc-transitive graphs is defined as follows:
	\begin{definition}
		Let $p$, $r$, $s$ be positive integers such that $p\ge 2$ and $r\ge 3$.
		Define a simple graph $\operatorname{C}(p,r,s)=(V,E)$ as follows:
		\begin{enumerate}[{\rm (i)}]
			\item  the vertex set $V$ is $\mathbb{Z}_r\times\mathbb{Z}^s_p$;\vskip0.1in
   
			\item  the edge set $E$ is defined to be the set of all pairs of the form \[\{(i,x_0,x_1,\ldots,x_{s-1}),(i+1,x_1,\ldots,x_{s-1},x_s)\}\] for every $i\in\mathbb{Z}_r$ and $x_0,x_1,\ldots,x_{s-1},x_s\in \mathbb{Z}_p$.
		\end{enumerate}
		The graph $\operatorname{C}(p,r,s)$ is called a \emph{Praeger-Xu graph}, or simply, a PX graph.
	\end{definition}
	
	We give the following notation, which will be used frequently.
	\begin{notation}
		Throughout this paper, we assume 
		$p$ is an odd prime and 
		$r\ge 3$
		is an integer, with 
		$p\nmid r$ unless otherwise specified.
	\end{notation}
	
	It is known that a PX graph $\operatorname{C}(p,r,s)$ is arc-transitive if and only if $r\ge s+1$ (see \cite[Theorem 2.10]{praeger1989Characterization}).

	The following theorem provides a criterion for determining whether an arc-transitive graph is a Praeger-Xu graph.
	\begin{theorem}[{\cite[Theorem 1]{praeger1989Characterization}}]\label{thm:PX-equiv}
		Let $\Gamma$ be a connected, simple, $G$-arc-transitive  graph of valency $2p$.
		If $G$ contains an abelian normal $p$-subgroup $M$ which is not semiregular on the vertices of $\Gamma$, then $\Gamma=\operatorname{C}(p,r,s)$ for some $r\ge\max\{3,s+1\}$ and $s\ge 1$.
	\end{theorem}
	
	\subsection{Rotary maps}\label{sec:arc-regular-maps}
	Let $\mathcal{M}=(V,E,F)$ be a map and let $G\le\Aut(\mathcal{M})$.
	
	Then $\mathcal{M}$ is \textit{$G$-arc-transitive} (respectively, $G$-arc-regular) if $G$ acts transitively (respectively, regularly) on the arc set of $\mathcal{M}$. 
	Observe that, for any map, the stabilizers of vertices, edges, and faces in the automorphism group are cyclic or dihedral~\cite{Auto-ETmaps}. 
	Using the definitions from Li, Praeger, and Song \cite{LiPraegerS-A}, a map 
	$\mathcal{M}$ is 
	\begin{enumerate}[{\rm (i)}]
		\item \emph{$G$-vertex-rotary} if  $\mathcal{M}$ is $G$-arc-regular and the vertex stabilizers $G_v$ are cyclic for all $v\in V$.\vskip0.1in
  
		\item \emph{$G$-vertex-reversing} if it is $G$-arc-regular and the vertex stabilizers $G_v$ are dihedral for all $v\in V$.\vskip0.1in
  
		\item \emph{$G$-rotary} if it is $G$-vertex-rotary and the face stabilizers $G_f$ are cyclic for all $f\in F$.
	\end{enumerate}
	According to \cite[Definition~1.5]{LiPraegerS-A}, if $\mathcal{M}$ is $G$-rotary, then $\mathcal{M}$ is also rotary, and the converse is true as well (consider $G$ to be the orientation-preserving automorphism group, see~\cite[Section 2]{orireg-multiK_n}).


	Let $G$ be a group and let $(\rho,\tau)\in G\times G$ be a rotary pair of $G$.
	Recall that a rotary pair satisfies $G=\langle\rho,\tau\rangle$ and $|\tau|=2$.
	We construct an incidence configuration $\RM(G, \rho, \tau)$ as follows:
	\[  \mbox{vertex set $[G : \langle \rho \rangle]$, edge set $[G : \langle \tau \rangle]$ and face set $[G : \langle \rho\tau \rangle]$, }
	\]
	where two objects are incident if and only if their set intersection is non-empty. 
	Then $\RM(G,\rho,\tau)$ always yields a $G$-rotary map \cite[Construction~4.2]{LiPraegerS-A}.
	Now let $\mathcal{M}$ be a $G$-rotary map with a vertex $v$ and an edge $e$. 
	Let $\rho$ and $\tau$ be generators of the stabilizers $G_v$ and $G_e$, respectively. 
	Then by \cite{LiPraegerS-A}, $G = \langle \rho, \tau \rangle$ and $|\tau|=2$, hence $(\rho,\tau)$ is a rotary pair of $G$.
	As established in \cite[Construction~4.2]{LiPraegerS-A}, $\RM(G, \rho, \tau)$ is a $G$-rotary map isomorphic to $\mathcal{M}$.

	It is worth noting that, given a rotary pair $(\rho, \tau)$ of $G$, one can also form the incidence configuration $\BiRM(G,\rho,\tau)$, which shares the same underlying graph \cite[Constructon 4.4]{LiPraegerS-A}.

	As with $\RM(G,\rho,\tau)$, $\BiRM(G,\rho,\tau)$ gives rise to a family of $G$-arc-transitive maps, called \emph{$G$-bi-rotary} maps, which are precisely the $G$-vertex-rotary maps whose face stabilizers in $G$ are dihedral \cite[Proposition~4.5(b)]{LiPraegerS-A}.

        We say that a  map $\mathcal{M}$ is a \emph{PX map} if the underlying graph is a PX graph. 
	Consequently, classifying $G$-rotary PX maps  simultaneously classifies $G$-bi-rotary PX maps, thereby yielding a complete classification of all $G$-vertex-rotary PX maps.


	In Section \ref{sec}, we will prove that every rotary PX map is isomorphic to a direct product of its quotient maps, each of which remains a rotary PX map. 
	This reduces the classification problem to classifying irreducible rotary PX maps, see Definition~\ref{def:irr-PXmap}.
	Now, we formally define quotient rotary maps and direct products.
	\begin{definition}\label{def:quot-RotaMap}
		Let $\mathcal{M} = \RM(G, \rho, \tau)$. 
		For a normal subgroup $M \trianglelefteq G$ with $\rho , \tau \notin M$, 
		the map $\RM(G/M, \rho M, \tau M)$ is the \emph{quotient map} of $\mathcal{M}$~\cite{alge-quotient}, denoted $\mathcal{M}/M$.
		
		For $n$ rotary maps $\mathcal{M}_i = \RM(G_i, \rho_i, \tau_i)$ ($i = 1, \dots, n$), 
		let $H = \langle (\rho_1,\dots,\rho_n), (\tau_1,\dots,\tau_n) \rangle \leq \prod_{i=1}^n G_i$. 
		The \emph{direct product} $\prod_{i=1}^n \mathcal{M}_i$ is defined as 
		$\RM\left( H, (\rho_1,\dots,\rho_n), (\tau_1,\dots,\tau_n) \right)$.
	\end{definition}
		The following proposition establishes a criterion for isomorphism between rotary maps via their rotary pairs.
        \begin{proposition}\label{prop:iso-RotaMap}
            Two maps $\RM(G,\rho,\tau)$ and $\RM(H,\rho',\tau')$ are isomorphic if there is a group isomorphism $f:G\to H$ such that $f(\rho)=\rho'$ and $f(\tau)=\tau'$.
        \end{proposition}
	This property naturally extends to homomorphisms: 
	\begin{proposition}\label{lem:quotient-maps}
		For  $i\in\{1,2\}$, let $\mathcal{M}_i=\RM(G_i,\rho_i,\tau_i)$.
		Then $\mathcal{M}_2$ is isomorphic to a quotient map of $\mathcal{M}_1$ if and only if there exists a homomorphism $\psi:G_1\rightarrow G_2$ such that $\psi(\rho_1)=\rho_2$ and $\psi(\tau_1)=\tau_2$.
	\end{proposition}
	It is straightforward to check that direct products satisfy the commutative law and the associative law.
	The lemma below plays an important role in Section~\ref{sec}.
	\begin{lemma}\label{lem:direct-product-rotary-map}
		Let $\mathcal{M}=\RM(G,\rho,\tau)$ and let $V_1,\ldots,V_n$ be normal subgroups of $G$ with $\bigcap_iV_i=1$. 
		Then $\mathcal{M}\cong \prod_{i=1}^n\mathcal{M}/V_i$.
	\end{lemma}
	
	\begin{proof}
		Define a function $\Lambda:G\rightarrow \prod_{i=1}^n G/V_i$ by $g\mapsto (gV_1,\ldots,gV_n)$. 
		It is straightforward to check that $\Lambda$ is a homomorphism.
		Since $\bigcap_i V_i = 1$, $\Lambda$ is injective and is thus an isomorphism $G\cong\Lambda(G)$.
		
		Since $\Lambda(G)=\l \Lambda(\rho),\Lambda(\tau) \r$, we have that 
		\[\prod_{i=1}^n\mathcal{M}/V_i=\RM(\Lambda(G),\Lambda(\rho),\Lambda(\tau)).\] 
		It follows from $\Lambda(G)\cong G$ that $\mathcal{M}$ is isomorphic to $\prod_{i=1}^n\mathcal{M}/V_i$.
	\end{proof}

	Since each rotary pair generates $G$, $\Aut(G)$  acts semiregularly on the set of rotary pairs of $G$.
	It follows that the number of isomorphism classes of $G$-rotary maps equals the number of $\Aut(G)$-orbits on the set of rotary pairs of $G$.
	In this work, maps are considered up to isomorphism. 
	Consequently, the set of $G$-rotary maps is in bijection with the set of $\mathrm{Aut}(G)$-orbits on rotary pairs of $G$, and thus their enumerations coincide.
	The above discussion is foundational to enumerating rotary PX maps in Section~\ref{sec:irr-PX-map}.

	\subsection{Representation theory of dihedral groups}
	In this subsection, we give some fundamental properties of irreducible representations of dihedral groups over $\mathbb{F}_p$. For completeness, we include the proofs in Appendix~\ref{sec:appendix}.
 
	Denote by $\operatorname{Irr}(\operatorname{D}_{2r})$ the set of 
	isomorphism classes of irreducible representations of $\operatorname{D}_{2r}$ over $\mathbb{F}_p$.
	The automorphism group of $\operatorname{D}_{2r}$ has a natural action on $\operatorname{Irr}(\operatorname{D}_{2r})$, which is called the \emph{$\Aut(\D_{2r})$-conjugation} of irreducible representations of $\D_{2r}$.
	Given a representation $\psi\in\operatorname{Irr}(\D_{2r})$, we denote the $\Aut(\D_{2r})$-orbit of $\psi$ by $\psi^{\Aut(\D_{2r})}$ (see \cite[Chapter 4.1]{FinGrp-Gor}).
	We often use the language of module theory to describe representations.

	\begin{definition}
		Let $D=\operatorname{D}_{2r}$ be a dihedral group and let $F$ be a field.
		Write $D=\langle a,b\rangle$ where $a,b\in D$ are involutions.
		Let $1$ denote the unit element of $F$.
		Define a representation $\gamma_{-1,-1}:D\rightarrow F^\times$ of degree $1$ by $\gamma_{-1,-1}(a)=-1$ and $\gamma_{-1,-1}(b)=-1$. Let $\gamma_{1,1}$ be the trivial representation.
		If $r$ is even, then define a representation $\gamma_{1,-1}:D\rightarrow F^\times$ of degree $1$ by $\gamma_{1,-1}(a)=1$ and $\gamma_{1,-1}(b)=-1$
		and a representation $\gamma_{-1,1}:D\rightarrow F^\times$ of degree $1$ by $\gamma_{-1,1}(a)=-1$ and $\gamma_{-1,1}(b)=1$.
	\end{definition}

	We use the notation $E_\lambda(A)$ to denote the eigenspace of a linear operator $A$ corresponding to the eigenvalue $\lambda$.
	Let $\psi:G\rightarrow \operatorname{GL}(V)$ be a representation of group $G$ and let $g\in G$.
	For convenience, we usually write $E_\lambda(g)$ rather than $E_\lambda(\psi(g))$ if the context is clear.
	\begin{lemma}\label{lem:irreducible-representation-dihedral}
		Let $D=D_{2r}$ and
		let $z\in D$ be an involution which does not lie in the center of $D$.
		Let $\psi:D\rightarrow\operatorname{GL}(V)$ be an irreducible representation of $D$ of degree $d$ over $\mathbb{F}_p$.
		\begin{enumerate}[{\rm (i)}]
			\item  If $d=1$ and $r$ is odd, then $\psi\in\{\gamma_{1,1},\gamma_{-1,-1}\}$.\vskip0.1in
   
			\item  If $d=1$ and $r$ is even, then $\psi\in\{\gamma_{1,1},\gamma_{-1,-1},\gamma_{1,-1},\gamma_{-1,1}\}$.\vskip0.1in
   
			\item  If $d\ge 2$, then $d$ is even and $\dim(E_{1}(z))=\dim(E_{-1}(z))=d/2$.
		\end{enumerate}
	\end{lemma}
	
	\begin{lemma}\label{lem:module-homomorphism}
		Let $D=\operatorname{D}_{2r}=\langle c\rangle\rtimes\langle b\rangle$ and let $M$ be a simple $D$-module over the finite field $\mathbb{F}_p$ of even dimension $d$.
		
		Then
		\begin{enumerate}[{\rm (i)}]
			
			\item $\operatorname{End}_D(M)\cong \mathbb{F}_{p^{d/2}}$; in particular, $(\operatorname{End}_D(M))^\times\cong \mathbb{Z}_{p^{d/2}-1}$;\vskip0.1in
   
			\item for every $0\ne m,n\in E_1(b)$, there exists $f\in \operatorname{End}_D(M)$ such that $f(m)=n$.  
		\end{enumerate}
	\end{lemma}

	Let $C=\mathbb{Z}_r$ be the cyclic subgroup generated by $c$.
	Let $f(x)\in\mathbb{F}_p[x]$ be an irreducible polynomial such that $f(x)\mid x^r-1$.
	The map $C\rightarrow \mathbb{F}_p[x]/(f(x)), c\mapsto x$ induces an irreducible representation of $C$ on $\mathbb{F}_p[x]/(f(x))$
	by multiplication.
	It is well known that this gives a bijection between irreducible
	representations of $C$ and irreducible factors of $x^r-1$ (see, for example, \cite[Chapter 6]{vanlint1999Introduction}).
	Using this fact, we show that the faithful irreducible representations of $\operatorname{D}_{2r}$
	share the same degree. 
	
	\begin{lemma}\label{lem:faith-rep}
		Let $D=\operatorname{D}_{2r}=\langle c\rangle\rtimes\langle b\rangle$. Let $C=\langle c\rangle$ and let $d$ be the degree of a faithful irreducible representation of $C$ over $\mathbb{F}_p$. 
		Then $d$ is the least positive integer such that $r\mid p^d-1$.
            If $\psi$ is a faithful irreducible representation of $D$ over $\mathbb{F}_p$,
            then \[\operatorname{deg}(\psi)=\left\{\begin{aligned}
            d&,\ \mbox{$d$ is even and $p^{d/2}\equiv -1\pmod{r}$},\\
            2d&,\ \mbox{otherwise}.
        \end{aligned}\right.\]
        Moreover, there are $\varphi(r)/\operatorname{deg}(\psi)$ such representations.
	\end{lemma}

	\section{Construction of rotary PX maps}\label{sec:section3} 
	In this section, we assume  $G=V\rtimes_\psi D=\ZZ_p^d\rtimes_\psi\D_{2r}$ where $d\geq 2$. Let $\D_{2(p-1)}=\l c \r\rtimes\l b \r=\ZZ_{p-1}\rtimes\ZZ_2$. For each $v\in V$ and $x\in D$, we define $v^x=\psi(x)^{-1}(v)$.

	Now we construct a family of $G$-rotary PX maps for the case where $\psi$ is irreducible.
	
	\textbf{Construction:}\label{Cons:PX-map} 
	Suppose that $\psi$ is irreducible.
	Let $x,y\in D$ be involutions such that $D=\langle x,y\rangle$,  and let $1\ne v\in\operatorname{C}_V(x)$. (By Lemma~\ref{lem:irreducible-representation-dihedral} {\rm (iii)}, $\operatorname{C}_V(x)\ne 0$.)
	We  define a rotary map $\RM(G,vx,y)$.
	Note that $|vx|=2p$. 
	As we shall see in Section~\ref{sec:characterization}, $\RM(G,vx,y)$ is a rotary PX map (see Theorem~\ref{thm:PX-equiv} and Theorem~\ref{thm:arc-regular-PX-graph}).
	
 The example below illustrates the above construction.
	\begin{example}\label{irreducible-PX-map-example}
		Let $r=p-1$,  and let $\omega$ be a primitive $r$-th root of unity over $\mathbb{F}_p$.  
		Let $d=2$ and let $\psi$ be given by 
		$$\psi(c)=\begin{bmatrix}
			\omega^{-1}&0\\
			0&\omega
		\end{bmatrix}\text{ and }
		\psi(b)=\begin{bmatrix}
			0&1\\
			1&0
		\end{bmatrix}.$$
		Let $v=(\omega,-1)^{\top}\in V$.
		Then $v\in \operatorname{C}_V(cb)$ and $G$ is generated by $vcb$ and $b$ as $\psi$ is irreducible. 
		Let $\rho=vcb$ and $\tau=b$.
		Then $(\rho,\tau)$ is a rotary pair of $G$.
            Let $\Gamma$ be the underlying graph of $\RM(G,\rho,\tau)$.
            Recall that $G\le\Aut(\Gamma)$ and $G_\alpha=\l\rho\r$.
		Since $\l\rho\r\cap V=\ZZ_p$, we have that $\Gamma$ is a PX graph by Theorem~\ref{thm:PX-equiv}. 
		Therefore, $\RM(G,\rho,\tau)$ is a rotary PX map with face length $|\rho\tau|=|vc|=p-1$.
	\end{example}

 Now we give some properties about our construction.
	\begin{proposition}\label{prop:equivalence-ggMap}
		Using the notation in Construction~\ref{Cons:PX-map}, let $1\ne v'\in\operatorname{C}_V(x)$. 
		Then  $\RM(G,vx,y)\cong\RM(G,v'x,y)$.
	\end{proposition}
	\begin{proof}
		By Lemma~\ref{lem:quotient-maps}, $\RM(G,vx,y)\cong\RM(G,v'x,y)$ if and only if there exists an automorphism $f$ of $G$ such that $f(vx)=v'x,f(y)=y$.
		Since $v,v'\in \mathrm{C}_V(x)$, by Lemma~\ref{lem:module-homomorphism} (ii), there exists an $l\in (\mathrm{End}_D(V))^\times$ such that $l(v)=v'$.
		
		Let $f$ be a function given by
		$f(wg)=l(w)g$ for every $w\in V$, $g\in D$.
		We claim that $f$ is an automorphism. Suppose that $w_1g_1,w_2g_2$ are two elements of $G$. Then 
		$$\begin{aligned}
			f(w_1g_1w_2g_2) & =f(w_1w_2^{g_1^{-1}}g_1g_2)\\
			& =l(w_1)l(w_2^{g_1^{-1}})g_1g_2\\
			& =l(w_1)(l(w_2))^{g_1^{-1}}g_1g_2\\
			& =l(w_1)g_1l(w_2)g_2\\
			& =f(w_1g_1)f(w_2g_2).
		\end{aligned}$$
		This implies that $f$ is bijective; therefore, $f$ is an automorphism.   
		
		It follows that $f(vx)=v'x$ and $f(y)=y$ which implies that $\RM(G,vx,y) \cong \RM(G,v'x,y)$.
	\end{proof}
	Given an involution $x$ of the group $G=V\rtimes_\psi D=\ZZ_p^d\rtimes_\psi\D_{2r}$, we denote by $\rho_x$ an arbitrary element in the set $\{ vx\ |\ 1\neq v\in \C_V(x) \}$. 
	For the case where $\psi$ is irreducible, we show that all $G$-rotary PX maps can be constructed in this way.
    Specifically, for every rotary map $\RM(G,\rho,\tau)$ such that $|\rho|=2p$, there exists $x,y\in D$ such that $\RM(G,\rho,\tau)\cong\RM(G,\rho_x,y)$
	(see Corollary~\ref{coro:isomorphic-ggMap}). 
	In addition, rotary PX maps can be constructed using this method when $p\mid r$.
    However, the resulting maps might not cover all possible isomorphism classes under this condition.
	
	At the end of this part, we provide an example of a $G$-rotary PX map
	where $\psi$ is reducible.
	
	\begin{example}\label{examp:reducible}
		Assume that $r\ge 5$ is a primitive prime divisor of $p^2-1$.
		For example, $(p,r)=(11,5)$.
		Then $r$ is an odd prime and $r\mid p+1$.
		Therefore, every representation of $\operatorname{D}_{2r}$ of degree greater than $1$ are faithful.
		By Lemma~\ref{lem:faith-rep}, there are $(r-1)/2$ different irreducible representations of $\operatorname{D}_{2r}$ of degree $2$. 
		
		Note that $(r-1)/2\ge 2$.
		Let $\psi_1$ and $\psi_2$ be two irreducible representations of $\operatorname{D}_{2r}$ of degree $2$ such that $\psi_1\not\cong\psi_2$ and let
		$G=(V_1\times V_2)\rtimes_{\psi_1\times\psi_2}\operatorname{D}_{2r}$ 
		where $\operatorname{D}_{2r}=\langle c \rangle\rtimes\langle b \rangle=\mathbb{Z}_{r}\rtimes\mathbb{Z}_2$.
		Let $1\ne v_1\in \operatorname{C}_{V_1}(cb)$ and let 
		$1\ne v_2\in\operatorname{C}_{V_2}(cb)$.
            By an analogous method in the proof of Theorem~\ref{thm:existence-decomposition},
            one can derive that
		$G$ is generated by $v_1v_2cb$ and $b$.
		Let $\rho=v_1v_2cb$ and let $\tau=b$.
            Let $\Gamma$ be the underlying graph of $\RM(G,\rho,\tau)$.
            Recall that $G\le\Aut(\Gamma)$ and $G_\alpha=\l\rho\r$.
		Since $\l\rho\r\cap V=\ZZ_p$, we have that $\Gamma$ is a PX graph by Theorem~\ref{thm:PX-equiv}. 
		Therefore, $\RM(G,\rho,\tau)$ is a rotary PX map.
		We point out that if $\psi_1\cong\psi_2$, then $\langle \rho,\tau\rangle\cong V_1\rtimes_{\psi_1}\operatorname{D}_{2r}\not\cong G$,
		so the condition $\psi_1\not\cong\psi_2$ is essential here.
		
		We now present an alternative method for constructing a $G$-rotary PX map.
		Observe that for every $i\in\{1,2\}$, $G/V_i\cong V_i\rtimes_{\psi_i}\operatorname{D}_{2r}$.
		Let $G_i=V_i\rtimes_{\psi_i}\operatorname{D}_{2r}$ for every $i\in\{1,2\}$.
		Let $1\ne v_i\in\operatorname{C}_{V_i}(ab)$, 
		let $\rho_i=v_icb$ and let $\tau_i=b$.
		As we discussed above, $(\rho_i,\tau_i)$ is a rotary pair of $G_i$, and $\RM(G_i,\rho_i,\tau_i)$ is a rotary PX map.
		Let $H=\langle (\rho_1,\rho_2),(\tau_1,\tau_2)\rangle\le G_1\times G_2$.
		Then one can check that $V_1\times V_2\le H$ and $\langle(\rho_1,\rho_2)\rangle\cap (V_1\times V_2)\cong\mathbb{Z}_p$.
		Therefore, $\RM(H,(\rho_1,\rho_2),(\tau_1,\tau_2))$ is a rotary PX map.
		Also note that $H\cong G$.
		This gives a $G$-rotary PX map as well.
	\end{example}

	As illustrated in Example~\ref{examp:reducible}, the case where $\psi$ is reducible can often be simplified to the irreducible case using the direct product of rotary maps (see details in Theorem~\ref{fenjie}).
	Although with some exceptions, this gives all the $G$-rotary PX maps.

	\section{Characterization of rotary PX maps}\label{sec:characterization}
	
	In this section, we characterize rotary PX maps. We provide a more general characterization for all arc-regular maps in Theorem~\ref{thm:arc-regular-PX-graph}.
	
	We now introduce a generalized definition of PX graphs. This is necessary because, as shown in Lemma~\ref{lem:G-AT-QuoMap}, the quotient of a rotary PX map is not always a rotary PX map. The following definition ensures that our class of maps is closed under the quotient operation.
	\begin{definition}\label{def:augPXgraph}
		Let $s\geq 0$ be an integer, and let $\delta\in\{1,-1\}$.
		We define an \emph{augmented PX graph} $\operatorname{C}^*(p,r,s,\delta)$ as follows:
		\begin{enumerate}[{\rm (i)}]
			\item  $\operatorname{C}^*(p,r,s,\delta)=\operatorname{C}(p,r,s)$ if $s\ge 1$;\vskip0.1in
   
			\item $\operatorname{C}^*(p,r,0,1)=\C^{(p)}_r$ is a multicycle of length $r$ in which every pair of adjacent vertices is joined by $p$ parallel edges; \vskip0.1in
   
			\item [\rm (iii)] $\operatorname{C}^*(p,r,0,-1)=\operatorname{C}_{pr}$.
		\end{enumerate}
		The integer \( r \) is called the \emph{length} of \( \operatorname{C}^*(p,r,s,\delta) \). 
		For convenience, we use \( \operatorname{C}^*(p,r,s) \) instead of \( \operatorname{C}^*(p,r,s,\delta) \) 
		when \( s \geq 1 \) or \( \delta = 1 \).
		An \emph{augmented PX map} of length $r$ is a map  whose underlying graph is an augmented PX graph of length $r$.
	\end{definition}
	
	It is clear that $\Aut(\operatorname{C}^*(p,r,0))=\operatorname{S}_p\wr_{\Omega}\operatorname{D}_{2r}$
	where $\Omega$ is the vertex set of $\operatorname{C}^*(p,r,0)$ and $\operatorname{D}_{2r}$ acts naturally as the automorphism group of $\operatorname{C}_r$.
	
	Recall that $G$-rotary maps are $G$-arc-regular, in particular, $G$ is an arc-regular subgroup of $\Aut(\operatorname{C}^*(p,r,s,\delta))$.
	We claim the following theorem which will be proved later. 
 Recall that $p$ is an odd prime, $r\geq 3$ and $p\nmid r$.
	\begin{theorem}\label{thm:arc-regular-PX-graph}
		Let $s\geqslant 0$ be an integer, let $\delta\in\{1,-1\}$ and let $r$ be an integer such that $r\ge \max\{3,s+1\}$.
		Let $G$ be an arc-regular group of automorphisms of $\operatorname{C}^*(p,r,s,\delta)$.
		Then $G\cong \mathbb{Z}_p^{s+1}\rtimes\operatorname{D}_{2r}$.
	\end{theorem}  
	As a direct consequence of Theorem~\ref{thm:arc-regular-PX-graph}, we obtain the following corollary.
	\begin{corollary}\label{cor:aut-map}
		Let $\mathcal{M}$ be a $G$-rotary map with underlying graph isomorphic to $\operatorname{C}^*(p,r,s,\delta)$.
		Then $G$ is isomorphic to $\mathbb{Z}_p^{s+1}\rtimes\operatorname{D}_{2r}$.
	\end{corollary}
	
	According to this corollary, we give the following definition, which is essential in the following sections.
	\begin{definition}\label{def:irr-PXmap}

		A $G$-arc-regular augmented PX map $\calM$ is \emph{irreducible} if the group $G\cong\operatorname{Z}_p^{s+1}\rtimes_\psi \operatorname{D}_{2r}$ where $\psi$ is irreducible.
		
	\end{definition}

	The proof of Theorem \ref{thm:arc-regular-PX-graph} follows directly from Lemma~\ref{lem:ArcReg-Auto-PXGraph-2} and Lemma~\ref{lem:ArcReg-Auto-PXGraph-3}, which we establish next.
	Clearly, the arc-regular automorphism subgroup of $ \operatorname{C}^*(p,r,0,-1)$ is isomorphic to $\operatorname{D}_{2pr}=\mathbb{Z}_p\rtimes\operatorname{D}_{2r}$.
	Hence it is sufficient to consider $ \operatorname{C}^*(p,r,s,\delta)$ where $(s,\delta)\neq (0,-1)$, as in the following three lemmas.

	\begin{lemma}\label{lem:ArcReg-Auto-PXGraph-2}
		Let $\Gamma=\operatorname{C}(p,r,s)$, where $s\geq 0$, $r\geqslant \max\{s+1,3\}$.
		Suppose that $G$ is an arc-regular subgroup of $\Aut(\Gamma)$, and $(r,s)\neq (4,1)$.
		Then $G\cong \ZZ_p^{s+1}\rtimes\D_{2r}$.
	\end{lemma}
	\begin{proof}
		By \cite[Theorem 2.13]{praeger1989Characterization},  $G\le\operatorname{S}_p\wr_{\Omega}\operatorname{D}_{2r}$.
		Also note that the quotient graph $\Gamma_{\operatorname{S}_p^\Omega}=\operatorname{C}_r$, the cycle of length $r$.
		Let $K$ be the kernel of $G$ acting on $\C_r$.
		Then $K\geqslant\S_p^{\Omega}$ and
		$|G/G\cap \S_p^{\Omega}|\leqslant |G/G\cap K|$.	
		Let $\bar{G}=GK/K\le\Aut(\operatorname{C}_r)$ be the induced group of $G$.
		Since $G$ is arc-regular, the induced group $\bar{G}$ on $\operatorname{C}_r$ is also arc-transitive, which implies that $\bar{G}=\Aut(\operatorname{C}_r)=\operatorname{D}_{2r}$.
		Therefore, $\operatorname{S}_p\wr_\Omega\operatorname{D}_{2r}=G\operatorname{S}_p^{\Omega}$.
		And we have $2r=|G/(G\cap \operatorname{S}_p^{\Omega})|\leq|\bar{G}|=2r$ which implies that $G/(G\cap \operatorname{S}_p^{\Omega})\cong \D_{2r}$.
		Then $G\cap\operatorname{S}_p^\Omega$ is a $p$-group of order $p^{s+1}$ as $|G|=2rp^{s+1}$.
		Since a Sylow $p$-subgroup of $\operatorname{S}_p^{\Omega}$ is isomorphic to $\mathbb{Z}_p^\Omega$,
		we have that $G\cap\operatorname{S}_p^\Omega\cong\mathbb{Z}_p^{s+1}$.
		Therefore, $G\cong\mathbb{Z}_p^{s+1}.\operatorname{D}_{2r}$. 
		By the Schur-Zassenhaus theorem, this extension is split and $G\cong\mathbb{Z}_p^{s+1}\rtimes\operatorname{D}_{2r}$.
	\end{proof}

	Now we deal with cases where $(r,s)=(4,1)$.

	\begin{lemma}\label{lem:ArcReg-Auto-PXGraph-3}
		Let $\Gamma=\operatorname{C}(p,4,1)$.
		Suppose that $G$ is an arc-regular subgroup of $\Aut(\Gamma)$.
		Then $G$ is isomorphic to $\ZZ_p^{2}\rtimes\D_{8}$.
	\end{lemma}
	
	\begin{proof}
		We have that $G$ is an arc-regular group on the complete bipartite graph $K_{2p,2p}$ according to \cite[Theorem 2.13]{praeger1989Characterization}.
		It follows that $|G|=8p^2$.

		Let $U$ and $V$ be the two parts of $K_{2p,2p}$, let $\alpha\in U$ and let $\beta\in V$.
		Then there exists a unique involution $z\in G$ such that $z(\alpha)=\beta$ and $z(\beta)=\alpha$, as $G$ is arc-regular.
		Therefore, $G_\alpha^z=G_\beta$.
		Let $G^+$ be the subgroup of $G$ which consists of automorphisms fixing $U$ and $V$.
		Then $G^+$ is a normal subgroup of $G$ and $G=G^+\rtimes\langle z\rangle$.
		Since $G_\alpha$ is transitive on $V$, we have that $G^+=G_\alpha G_\beta$ by the Frattini argument.
		Since $G$ is arc-regular, we have that $G_\alpha$ is regular on $V$.
		Therefore, $G_\beta\cong\mathbb{Z}_{2p}\mbox{ or }\operatorname{D}_{2p}$.
		Note that $G_\alpha\cong G_\beta$. We also have $G_\alpha\cong\mathbb{Z}_{2p}\mbox{ or }\operatorname{D}_{2p}$.
		Therefore, $G^+$ has at least two subgroups of order $p$, which implies that the Sylow $p$-subgroup of $G^+$ is isomorphic to $\mathbb{Z}_p^2$.

		We claim that $G^+$ has a unique Sylow $p$-subgroup $(G^+)_p$,
		hence $(G^+)_p$ is characteristic in $G^+$.
		By Sylow's theorem, if $p\ge 5$, then $G^+$ has a unique Sylow $p$-subgroup $(G^+)_p$.
		Now suppose that $p=3$ and 
		there is more than one Sylow $p$-subgroup of $G^+$.
		Note that a group of order $p^2$ is abelian.
		By Brodkey's theorem (\cite[Theorem 1.37]{isaacs2008Finite}), there are two Sylow $p$-subgroups $P$ and $Q$ of $G^+$ such that $P\cap Q=\operatorname{O}_p(G^+)$.
		Since $P$ is not normal in $G$, we have that $P\ne Q$.
		Observe that $|P||Q|/(|P\cap Q|)=|PQ|\le |G^+|=36$. 
		We deduce that $|\operatorname{O}_p(G^+)|=|P\cap Q|=3$.
		Let $H:=G^+/\operatorname{O}_p(G^+)$.
		Then $H$ is a group of order $12$ and $\operatorname{O}_p(H)=1$.
		
		By the classification of groups of order $12$, $H\cong\operatorname{A}_4$.
		Recall that $G^+=G_\alpha G_\beta$.
		Let $H_\alpha$ and $H_\beta$ be quotients of $G_\alpha$ and $G_\beta$
		with respect to $\operatorname{O}_p(G^+)$ respectively.
		Then $H=H_\alpha H_\beta$ and $H_\alpha,H_\beta\in \{\mathbb{Z}_2,\mathbb{Z}_{2p} ,\D_{2p}\}$.
		Since $\operatorname{A}_4$ does not have subgroups isomorphic to $\mathbb{Z}_{2p}$ or $\D_{2p}$, we have that $H_\alpha\cong H_\beta\cong \mathbb{Z}_2$. Then $12=|H|=|H_\alpha H_\beta|\le 4$, a contradiction.
		Thus, the claim is proved.

		Note that $(G^+)_p$ is also a Sylow $p$-subgroup of $G$.
		Write $(G^+)_p=G_p$.
		Then, $G_p$ is also normal in $G$.
		Let $\bar{G}=G/G_p$. 
		We have $\bar{G}\cong (\bar{G}_\a\bar{G}_\b)\rtimes\langle\bar{z}\rangle$.
		Since $\bar{G}_\alpha\cong\bar{G}_\beta\cong\mathbb{Z}_2$ and $\bar{G}\alpha^{\bar{z}}=\bar{G}_\beta$,
		we have that $\bar{G}\cong \mathbb{Z}_2^2\rtimes\mathbb{Z}_2\cong\operatorname{D}_8$.
		Therefore, $G\cong\mathbb{Z}_p^2\rtimes\operatorname{D}_8$. 
		This completes the proof.
	\end{proof}

	When $p\mid r$, we still have $G=\ZZ_p^{s+1}.\D_{2r}$ but the extension might not be split (see the example below).
	
	\begin{example}\label{example:p|r}
		Let $$G=\l d\r\times(\l a\r\rtimes(\l b\r\times\l c\r))\cong \ZZ_2\times (\ZZ_{p^2}\rtimes(\ZZ_2\times\ZZ_p))$$ 
		where $a^b=a^{-1}$ and $a^c=a^{1-p}$.
		
		We claim that $G\cong \ZZ_p^2.\D_{4p}$. Let $M=\l a^p,c\r$. Since $a^c=a^{1-p}$, we have that 
		$$\begin{aligned}
			&(a^{p})^{c}=a^{p-p^2}=a^p \text{ and }\\
			&c^a=a^{-1}c a=c(a^{-1})^{c}a=ca^{p-1}a=ca^p.
		\end{aligned}$$
		Note that $M^d=M$ and $M^b=M$. We have that $\ZZ_p^2\cong M\trianglelefteq G$. Let $\Bar{G}=G/M$. We now show that $\Bar{G}\cong \D_{4p}$. Since $G=\l d\r\times(\l a\r\rtimes(\l b\r\times\l c\r))$, we have that 
		$$\Bar{G}=\l \Bar{d}\r\times (\l \Bar{a}\r\rtimes\l \Bar{b}\r)\cong \ZZ_2\times \D_{2p}\cong \D_{4p}.$$

		We claim that the extension of $G\cong \ZZ_p^2.\D_{4p}$ is non-split. First, one can check that $M$ is the unique elementary abelian normal $p$-subgroup of order $p^2$. 
		Now suppose that this extension is split, which means that there exists a subgroup $H\cong \D_{4p}$ such that $H\cap M=1$. 
        Assume that $h\in H$ has order $p$. 
        Note that $M$ consists of elements of order $p$.
        We have that $h\in M$.
        This contradicts the condition that $H\cap M=1$.
		
		Let $\rho=bc$ and $\tau=dab$, then $|\rho|=2p$ and $|\tau|=2$. We claim that $\l \rho,\tau\r=G$. Let $K=\l \rho,\tau\r$. Since $\rho=bc$ and $[b,c]=1$, we obtain that $\rho^2=c^2$ and $\rho^p=b$ which implies $b,c\in K$. Note that $(\tau b)^2=dada=a^2\in K$, we conclude that $a\in K$. It follows that $d=\tau b a^{-1}\in K$. Therefore, $K=\l \rho,\tau\r=G$.
		
		Now, we construct a rotary map $\mathcal{M}$ by $(G,\rho,\tau)$. Let \[\mathcal{M}=\mathrm{RotaMap}(G,\rho,\tau)\] and 
		let $\Gamma$ be the underlying graph of $\mathcal{M}$. 
		Recall that $\ZZ_p^2\cong M=\l  a^p,c\r\trianglelefteq G$. By definition, the vertex stabilizer $G_v=\l \rho\r=\l bc\r$. Therefore, 
		$$G_v\cap M=\l \rho^2\r=\l c^2\r\cong \ZZ_p.$$ 
		By Theorem~\ref{thm:PX-equiv}, we conclude that $\Gamma$ is a PX graph. Consequently, $\mathcal{M}$ is a $G$-rotary PX map where $G=\ZZ_p^2.\D_{4p}$ is non-split.
	\end{example}

	We now establish a criterion to determine whether the underlying graph of an arc-regular map is a Praeger-Xu graph.
	Since rotary maps are arc-regular, this lemma provides a characterization of rotary augmented PX maps as well.
	
	\begin{lemma}\label{lem:arc-regular-PX-map}

		Let $G = \mathbb{Z}_p^{s+1} \rtimes \operatorname{D}_{2r}$ with $ s \geqslant 0 $, and let \( \mathcal{M} \) be a $G$-arc-regular map.
		If $|G_\a|=2p$, 
		then the underlying graph $\Gamma$ is isomorphic to  $ \operatorname{C}^*(p,r,s) $.
	\end{lemma}
	
	\begin{proof}
		Let $(\alpha,e,\beta)$ be an arc of $\Gamma$.
		Since $\mathcal{M}$ is $G$-arc-regular, we have that $\Gamma$ is $G$-arc-regular.
		Therefore, the vertex valency of $\Ga$ equals $|G_\a| = 2p$.
		Recall that $p\nmid r$.
		We have that $V\cong \ZZ_p^{s+1}$ is the unique normal Sylow $p$-subgroup of $G$ which  implies $|G_\alpha\cap V|=p$.
		Therefore $V$ is not semiregular on the vertices of $\Gamma$.

		By \cite[Theorem 1.2]{LiPraegerS-A}, the edge multiplicity of $\Gamma$ is $|G_{\alpha\beta}|$.
		First, if $|G_{\alpha\beta}|=1$, then $\Gamma$ is a  $G$-arc-regular simple graph of valency $2p$.
		Since $V$ is not semiregular on the vertices of $\Gamma$, we have that $\Gamma=\operatorname{C}(p,r,s)$ by Theorem~\ref{thm:PX-equiv}.
		
		Let $\tau$ be the generator of $G_e$.
		Then $(G_\a)^{\tau}=G_\beta$, and so $G_{\alpha\beta}=G_\alpha\cap (G_\alpha)^\tau$.
		Thus $G_{\alpha\beta}^\tau=G_{\alpha\beta}$.
		
		Now, suppose that $|G_{\alpha\beta}|=2$.
		Take the generator $\tau_1$ of $G_{\alpha\beta}$.
		Then by $G_{\alpha\beta}^\tau=G_{\alpha\beta}$, $\tau$ centralizes $\tau_1$.
		Since $|G_\a|=2p$ and $|G_e|=2$, we have that $G_\alpha V/V=\langle \tau_1 V\rangle$ and $G_e V / V=\langle \tau V \rangle$.
		Then $G / V =\l G_\a, G_e\r/V= \langle \tau_1 V, \tau V \rangle \cong \mathbb{Z}_2^2$.
		Also note that $ G / V \cong \operatorname{D}_{2r}$ which implies that $r=2$.
        This contradicts to our assumption that $r\ge 3$.

		Assume  that $|G_{\alpha\beta}|=p$, equivalently, the edge multiplicity of $\Gamma$ is $p$.
		Then the simple base graph of $\Gamma$ is of valency $2$. 
		Therefore, $\Gamma$ is a cycle with
		$p$ multiple edges between each adjacent pair of vertices.
		Hence $\Gamma=\operatorname{C}^*(p,r_1,0)$ for some $r_1\ge 3$.
		By Theorem~\ref{thm:arc-regular-PX-graph}, $G\cong\mathbb{Z}_p\rtimes\operatorname{D}_{2r_1}$,
		which implies that $r=r_1$ and $\Gamma=\operatorname{C}^*(p,r,0)$.
		
		If $|G_{\alpha\beta}|=2p$, then $G_\alpha=G_{\alpha\beta}$ which implies that $\tau$ normalizes $G_{\alpha}$.
		By $G=\l G_\a,\tau\r$, $G=G_\alpha\rtimes\langle\tau\rangle$, and so $|G|=4p$.
		Since $G=\mathbb{Z}_p^{s+1}\rtimes\operatorname{D}_{2r}$, we have that $s=0$ and $r=2$.
		However, this contradicts our assumption that $r\ge 3$. This completes the proof.
	\end{proof}

	\section{Irreducible PX maps}\label{sec:irr-PX-map}
	In this section, we study irreducible rotary augmented PX maps. 
	Throughout this section, let \( G =V\rtimes_\psi D= \mathbb{Z}_p^d \rtimes_{\psi} \operatorname{D}_{2r} \), 
	where \( \psi \) is irreducible. 
	We classify all \( G \)-rotary augmented PX maps for each such \( \psi \).  
	By Theorem~\ref{thm:arc-regular-PX-graph}, this yields a complete classification of irreducible rotary augmented PX maps.  
	Before proceeding, we first study the automorphism group of $G$.

	\subsection{Split affine extensions of a dihedral group}
	We study $\Aut(\mathbb{Z}_p^d\rtimes_{\psi}\operatorname{D}_{2r})$. 
	Recall that  $p$ is an odd prime and $p\nmid r$.
	
	The following notation will be used throughout this paper.
	Let $G$ be a group.
	For each $g \in G$, conjugation by $g$ induces the automorphism $\tilde{g}: h \mapsto ghg^{-1}$.
	The \textit{inner automorphism group} is $\mathrm{Inn}(G) := \{ \tilde{g} \mid g \in G \}$.
	In particular, for a subgroup $H\leqslant G$, we define $\Inn_G(H)=\{\tilde{h}\in \Aut(G) \ |\ h\in H\}$.
	With this notation, we have the following lemma.
	
	\begin{lemma}\label{lem:fram-Auto-AutM}
		Let $G=V\rtimes_\psi D$ where $V=\mathbb{Z}_p^d$, $p\nmid |D|$, and $\psi$ is irreducible and non-trivial.
		Set $A = \operatorname{Aut}(G)$. 
		Then \[A = \operatorname{Inn}_G(V) \rtimes \operatorname{N}_A(D)\] with $\operatorname{Inn}_G(V) \cong V$.
		
	\end{lemma}
	\begin{proof}
		Let $N=\operatorname{N}_G(D)$.
		We claim that $N=D$. 
		Since $N\ge D$, we have that
		$N=\operatorname{N}_V(D)D$.
		Since $V$ is normal in $G$, $\operatorname{N}_V(D)$ is normal in $N$.
		Also note that $D$ is normal in $N$ and $\operatorname{N}_V(D)\cap D=1$.
		Therefore, $\operatorname{N}_V(D)$ is centralized by $D$.
		Since $\psi$ is irreducible, $\operatorname{N}_V(D)=1\mbox{ or }V$.
		If $\operatorname{N}_V(D)=V$, then $\varphi$ is trivial.
		Hence $\operatorname{N}_V(D)=1$ and $N=D$.
		
		Note that for every $g\in A$, $g(D)$ is a complement of $V$. 
		By the Schur-Zassenhaus theorem and Frattini's argument, \[A=\operatorname{Inn}_G(V)\operatorname{N}_A(D).\]
		Since $\Inn_G(V)$ is the unique Sylow $p$-subgroup of $\Inn_G(G)$,
		we have that $\Inn_G(V)$ is a characteristic subgroup of $\Inn_G(G)$
		which implies that $\Inn_G(V)$ is normal in $A$.
		Hence, it is sufficient to show that $\operatorname{Inn}_G(V)\cap \operatorname{N}_A(D)=1$ and $\operatorname{Inn}_G(V)$ is normal in $A$.
		Let $1\ne l\in V$. If $\tilde{l}\in \operatorname{N}_A(D)$, 			
		then $l\in \operatorname{N}_G(D)=D$.
		Therefore, $l\le D\cap V=1$ which implies that $\operatorname{Inn}_G(V)\cap \operatorname{N}_A(D)=1$. 
		Since $N=D$, we have that $\operatorname{Z}(G)\le D$ which implies that $V\cap\operatorname{Z}(G)=1$.
		Note that $\operatorname{Inn}_G(V)\cong V/(V\cap\operatorname{Z}(G))$. We have that $\operatorname{Inn}_G(V)\cong V$. This completes the proof.
	\end{proof}
    Now we study the structure of $\operatorname{N}_{A}(\D_{2r})$ (see Lemma~\ref{Aut(D)_phi}).
    Before that, we introduce a useful lemma.
	\begin{lemma}\label{Iso-rep&groups-2}
		For $i\in \{1,2\}$, let $G_i=\ZZ_p^{d_i}\rtimes_{\psi_i}D_i$, where $p\nmid |D_1||D_2|$. 
		Then the following are equivalent:
		\begin{enumerate}[{\rm (i)}]
			\item $G_1\cong G_2$; \vskip0.1in
   
			\item there exists an isomorphism $\sigma:D_1\to D_2$ such that $\psi_2\circ\sigma\cong \psi_1$.
		\end{enumerate} 
		Moreover, if~{\rm (ii)} holds, there exists an isomorphism $f:G_1\rightarrow G_2$ such that $f(D_1)=D_2$ and $f|_{D_1}=\sigma$.
	\end{lemma}
	\begin{proof}
		Let $f:G_1\to G_2$ be an isomorphism.
		Then $d_1=d_2$. Write $d=d_1$.
		Since $p\nmid |D_1||D_2|$, we have that 
		there exists an element $v$ in the Sylow $p$-subgroup of $ G_2$ such that $f(D_1)^v=D_2$.
		Let $\sigma=\tilde{v}\circ f|_{D_1}$.
		Observe that $\sigma$ is an isomorphism from $D_1$ to $D_2$.
		Now we show that $\psi_2\circ\sigma\cong\psi_1$.

		For $i\in\{1,2\}$, let $ V_i $ be the unique Sylow $p$-subgroup of $G_i$.
		Then $V_1$ is a $D_1$-module with respect to $\psi_1$
		and $V_2$ is a $D_1$-module with respect to $\psi_2\circ\sigma$.
		Let $f|_{V_1}:V_1\rightarrow V_2$ be the restriction of $f$ on $ V_1 $.
		For each $ w\in V_1 $ and $ b\in D_1 $, we have that
		\[ \begin{aligned}
			f|_{V_1}(\psi_1(b)(w))=& f|_{V_1}(b^{-1}wb)\\
			=&f(b)^{-1}f(w)f(b)\\
		\end{aligned} \]
		Since $f(w)\in V_2$ and $V_2$ is abelian, 
		\[\begin{aligned}
			f(b)^{-1}f(w)f(b)&=(f(b)^{v})^{-1}f(w)f(b)^v\\
			&=\psi_2(f(b)^v)(f|_{V_1}(w))\\
			&=\psi_2\circ\sigma(f|_{V_1}(w))
		\end{aligned}\]
		Therefore, $f|_{V_1}$ serves as the isomorphism between representations $\psi_1$ and $\psi_2\circ\sigma$.

		Conversely, let $\sigma:D_1\to D_2$ be an isomorphism such that $\psi_1\cong \psi_2\circ \sigma$.
		For $i\in\{1,2\}$, let $ V_i $ be the unique Sylow $p$-subgroup of $G_i$.
		Then $V_1$ is a $D_1$-module with respect to $\psi_1$
		and $V_2$ is a $D_1$-module with respect to $\psi_2\circ\sigma$.
		Since $\psi_1\cong \psi_2\circ \sigma$, there exists a $D_1$-module isomorphism $f_0:V_1\to V_2$.
		Define $f:G_1\to G_2$ by 
		\[vb\mapsto f_0(v)\sigma(b),\]
		for each $v\in V_1$ and $b\in D_1$.
		Then $f(D_1)=D_2$ and $f|_{D_1}=\sigma$.
		Clearly, $f$ is bijective and one can check that $f$ is a homomorphism. 
		Hence $f$ is an isomorphism.
		This completes the proof.
	\end{proof}

	\begin{lemma}\label{Aut(D)_phi}
		Using notation in Lemma~\ref{lem:fram-Auto-AutM}, and let $N=\N_A(D)$.
		Define a homomorphism
		$\kappa:N\to\Aut(D)$ by $f\mapsto f|_{D}$.
		Then $\ker(\kappa )\cong \mathrm{End}_{D}(V)^\times$ and 
		\[\mbox{$\mathrm{im}(\kappa)= \Aut(D)_{\psi}$},\] 
		where $\Aut(D)_{\psi}=\{\sigma\in \Aut(D)\ \mid \ \psi\circ\sigma\cong \psi\}$.
	\end{lemma}
	\begin{proof}
		We first determine $\ker(\kappa)$.
		Let $f\in \ker(\kappa)$, i.e. the restriction $f\mid_{D}=\mathrm{id}_{D}$.
		It is clear that $f(V)=V$ as $V$ is characteristic in $G$.
		Note that $V$ is a $D$-module with respect to $\psi$.
		We claim that $f|_V$ is a $D$-module endomorphism.
		Let $v\in V$ and $b\in D$.
		We have $f|_V(\psi(b)(v))=f(v^b)=f(v)^{f(b)}=\psi(f(b))(f(v))$.
		It follows from $f\mid_{D}=\mathrm{id}_{D}$ that  $\psi(f(b))(f(v))=\psi(b)(f(v))$, and so 
		\[f|_V(\psi(b)(v))=\psi(b)(f|_V(v)).\]
		The claim is hence proved.
		Moreover, $f|_V$ is a $D$-module isomorphism by  definition.
		Therefore, we can define $\xi:\ker(\kappa)\to \mathrm{End}_{D}(V)^{\times}$ by $f\mapsto f|_V$.
		Since $f\mid_{D}=\mathrm{id}_{D}$,  we have that the homomorphism $\xi$ is injective.
		We next show that $\xi$ is surjective.
		For each $g\in \mathrm{End}_{D}(V)^{\times}$, define $f_g:G\to G$ by $vb\mapsto g(v)b$, where $v\in V$ and $b\in D$.
		One can check that $f_g$ is a homomorphism, and it is clear that $f_g\in N$.
		We therefore prove that $\xi$ is an isomorphic since $\xi(f_g)=g$.

		Now, we determine the image of $\kappa$.
		
		Let $f\in N$ and let $\sigma=\kappa(f)$.
		Let $U_1$ be the $D$-module with respect to $\psi$
		and let $U_2$ be the $D$-module with respect to $\psi\circ \sigma$.
		This means that $U_1$ and $U_2$ share the same base space $V$ but they are different $D$-modules.
		We show that $f|_V$ serves as a $D$-module isomorphism from $U_1$ to $U_2$.
		For each $u\in U_1$ and $b\in D$, we have that
		\[ \begin{aligned}
			f|_{V}(\psi(b)(w))=& f|_{V}(b^{-1}wb)\\
			=&f(b)^{-1}f(w)f(b)\\
			=&\sigma(b)^{-1}f(w)\sigma(b)\\
			=&\psi(\sigma(b))f|_V(w).
		\end{aligned} \]
		Therefore, $\psi\circ\sigma\cong\psi$, which implies that $\sigma\in\Aut(D)_\psi$.

		Let $ \sigma\in \Aut(D)_{\psi} $.
		Lemma~\ref{Iso-rep&groups-2} shows that there exists an automorphism $f\in N$ such that $\kappa(f)=\sigma$.
		Therefore $\sigma\in\mathrm{im}(\kappa)$.
		Thus, we conclude that $\mathrm{im}(\kappa)=\Aut(D)_\psi$.
	\end{proof}
	
	\subsection{Number of irreducible PX maps}
	
	In this subsection, we count the number of irreducible augmented PX maps of length $r$ and use this result to establish the one-to-one correspondence between irreducible PX maps of length $r$ and irreducible representations of degree greater than $1$.
Since the set of $G$-rotary augmented PX maps is in bijection with the $\Aut(G)$-orbits on rotary pairs $(\rho,\tau)$ of $G$ where $|\rho|=2p$, the number of such rotary maps is determined simply by counting these orbits (see Lemma~\ref{lem-|RPairs|=|IrrReps|}). 
Therefore, the total number of irreducible augmented PX maps of length $r$ is found by summing these individual counts over all applicable groups $G$ (see Theorem~\ref{|RPairs|=|IrrReps|}).

	Throughout this subsection, let $a,b\in\D_{2r}$ be involutions generating $\D_{2r}$.

	\begin{lemma}\label{lem:number}
		Let $G=V\rtimes_\psi D=\ZZ_p^d\rtimes_\psi \D_{2r} $, where  $\psi$ is irreducible.
		Then $G$ has a rotary pair $(\rho,\tau)$ such that $|\rho|=2p$ if $\psi\ne\gamma_{-1,-1}$ and $|\rho|=2$ if $\psi=\gamma_{-1,-1}$. 
		Moreover, the number of such rotary pairs of $G$ is
		\begin{enumerate}[{\rm (i)}]
			\item  $(p-1)r\varphi(r)$ if $\psi=\gamma_{1,1}$;\vskip0.1in
			
			\item $(p-1)pr\varphi(r)/2$ if  $\psi=\gamma_{1,-1}\mbox{ or }\gamma_{-1,1}$;\vskip0.1in
			
			\item  $pr\varphi(pr)$ if $\psi=\gamma_{-1,-1}$;\vskip0.1in
			
			\item  $p^d(p^{d/2}-1)r\varphi(r)$ if $d\ge 2$.
		\end{enumerate}
	\end{lemma}
	\begin{proof}
		If $\psi=\gamma_{1,1}$, then $G = \ZZ_p \times \D_{2r}$.
		The set of elements of order $2p$ in $G$ is $\{ vx \mid v \in V \setminus \{1\},\ x \in D \text{ with } |x| = 2 \}$.
		This set has $(p-1)r$ elements when $r$ is odd or $(p-1)(r+1)$ elements when $r$ is even.
		For an element $\rho = vx \in G$ of order $2p$, given that $\l \rho, \tau \r = G$ for some involution $\tau$, the involution $x$ must not lie in $Z(D)$.
		Thus, there are always $(p-1)r$ elements of order $2p$ that can serve as the rotation $\rho$ in a rotary pair $(\rho, \tau)$ of $G$.
		Note that all involutions of $G$ are contained in the subgroup $\D$.
		For each such $\rho = vx$ with $x \notin Z(\D)$, an involution $\tau$ satisfying $\l \rho, \tau \r = G$ also satisfies $\l x, \tau \r = \D$.
		Therefore, there are $\varphi(r)$ choices for $\tau$ with a given $\rho$.
		Hence, there are $(p-1)r\varphi(r)$ rotary pairs of $G = \ZZ_p \times \D$, which proves statement~(i).
		
		Now, let $r$ be even and $\psi = \gamma_{1,-1}$ or $\gamma_{-1,1}$.
		First, assume $\psi = \gamma_{1,-1}$, which is equivalent to $\psi(a)=1$ and $\psi(b)=-1$. 
		There are $(p-1)r$ choices for $\rho$, as each order $2p$ element has the form $va^{(ab)^i}$ or $v(ab)^{r/2}$ for some $v \in \ZZ_p \setminus \{1\}$. 
		Given that $\l \rho, \tau \r = G$ with $|\tau|=2$, there are $p\varphi(r)/2$ choices for $\tau$, as $\tau$ has the form $wb^{(ab)^{i}}$ for some $w\in \mathbb{Z}_p$. 
		Thus, the number of rotary pairs of $G$ is $(p-1) p r \varphi(r)/2$. 
		Similarly, when $\psi = \gamma_{-1,1}$, the number of rotary pairs is also $(p-1) p r \varphi(r)/2$. 
		This completes case~(ii).

		
		
		

		If $\psi=\gamma_{-1,-1}$, then $G=\D_{2pr}$ is a dihedral group.
		Since $|\rho|=2$, we have that there are $pr$ ways to choose $\rho$. 
		Since $\l \rho,\tau\r=G$ and $|\tau|=2$, we conclude that we have $\varphi(pr)$ choices of $\tau$. 
		Therefore, the number of rotary pairs of $G$ is $pr\varphi(pr)$. 
		This completes case~(iii).
		

		Now, assume $d \geq 2$. 
		Let $c\in \D$ and let $v\in \mathbb{Z}_p^d$. 
		Note that $\rho = vc$ has order $2p$ if and only if $\psi(c)v\ne v^{-1}$.
		By Lemma~\ref{lem:irreducible-representation-dihedral}, the $(-1)$-eigenspace of $c$ has dimension $d/2$. 
		Therefore, there are $(p^d - p^{d/2})r$ choices for $\rho$. 
		Write $\tau = wz$ where $w \in \mathbb{Z}_p^d$ and $z \in \D$. 
		The condition $|\tau| = 2$ implies that $|z| = 2$ and $\psi(z)w = w^{-1}$, 
		giving $p^{d/2}$ choices for $w$ again by Lemma~\ref{lem:irreducible-representation-dihedral}.
		Since $\langle \rho, \tau \rangle = G$, there are $\varphi(r)$ choices for $z$. 
		Thus, the number of rotary pairs $(\rho, \tau)$ in case~(iv) is $p^d (p^{d/2} - 1) r \varphi(r)$.
	\end{proof}
	
	We point out that $G$ does not have a rotary pair $(\rho,\tau)$ such that $|\rho|=2$ when $\psi\ne\gamma_{-1,-1}$, and $G$ does not have a rotary pair $(\rho,\tau)$ such that $|\rho|=2p$ when $\psi=\gamma_{-1,-1}$.

	Now, we count the number of the $G$-rotary augmented PX maps.

	\begin{lemma}\label{lem-|RPairs|=|IrrReps|}
		
		Let $G=V\rtimes_\psi D=\ZZ_p^d\rtimes_\psi \D_{2r} $, where $\psi$ is irreducible.
		Then the number of $G$-rotary augmented PX maps is $1$ if $d=1$; $|\psi^{\Aut(\D_{2r})}|$ if $d\geqslant 2$.
	\end{lemma}
	\begin{proof}
		Let $A=\Aut(G)$ and let $N=\operatorname{N}_A(D)$.
		
		First we discuss the cases where $d=1$.
		Suppose that $\psi=\gamma_{1,1}$, equivalently, $G=\mathbb{Z}_p\times\operatorname{D}_{2r}$. 
		Then $\Aut(G)=\Aut(\mathbb{Z}_p)\times\Aut(\operatorname{D}_{2r})$ since $p\nmid r$, and so $|\Aut(G)|=(p-1)r\varphi(r)$ is same to the number of rotary pairs of $G$ by Lemma~\ref{lem:number} {\rm (i)}.
		Hence there is a unique $G$-rotary augmented PX map.
		Assume that $\psi=\gamma_{-1,-1}$ yielding that $G\cong\operatorname{D}_{2pr}$ is dihedral.
		Then $|\Aut(G)|=pr\varphi(pr)$ equals  the number of rotary pairs of $G$ by Lemma~\ref{lem:number} {\rm (iii)}.
		Again there is a unique $G$-rotary augmented PX map.
		If $\psi=\gamma_{1,-1}\mbox{ or }\gamma_{-1,1}$, then the number of rotary pairs of $G$ is $(p-1)pr\varphi(r)/2$, by Lemma~\ref{lem:number} {\rm (ii)}.
		The automorphism group is $A\cong \ZZ_p^d\rtimes N$, by Lemma~\ref{lem:fram-Auto-AutM}. Since $\Aut(G)$ acts semiregularly on the set of rotary pairs of $G$, the number of isomorphism classes is 
		\[(p-1)r\varphi(r)/2|N|. \] It is clear that $\operatorname{End}_{\operatorname{D}_{2r}}(\mathbb{Z}_p)^{\times}\cong\mathbb{Z}_{p-1}$.
		Since $\gamma_{1,-1}\circ\sigma=\gamma_{-1,1}\circ\sigma=\gamma_{-1,1}\mbox{ or }\gamma_{1,-1}$ for every $\sigma\in\Aut(\operatorname{D}_{2r})$, we have that $|\Aut(\operatorname{D}_{2r})_\psi|=|\Aut(\operatorname{D}_{2r})|/2$.
		Therefore, $|N|=(p-1)r\varphi(r)/2$. It follows that the number of the maps is $1$. 
		
		Now, we suppose that $d\ge 2$.
		The number of rotary pairs of $G$ is given by Lemma~\ref{lem:number} {\rm (iv)}, as below:
		\[p^d(p^{d/2}-1)r\varphi(r).\] 
		There is  $A\cong \ZZ_p^d\rtimes N$ according to Lemma~\ref{lem:fram-Auto-AutM}.
		Since $A$ acts semiregularly on the set of rotary pairs of $G$, the number of the maps is 
		\begin{equation*}\label{equation:number}
			(p^{d/2}-1)r\varphi(r)/|N|.
		\end{equation*}
		By Lemma~\ref{Aut(D)_phi}, we have that $N/\mathrm{End}_{\D_{2r}}(\ZZ_p^d)^\times\cong \Aut(\D_{2r})_\psi$.
		By Lemma~\ref{lem:module-homomorphism} {\rm (i)}, $\mathrm{End}_{\D_{2r}}(\ZZ_p^d)^\times\cong \ZZ_{p^{d/2}-1}$.
		It follows that the number of the maps is
		\[r\varphi(r)/|\Aut(\D_{2r})_\psi|=|\Aut(\D_{2r})|/|\Aut(\D_{2r})_\psi|=|\psi^{\Aut(\D_{2r})}|.\] 
	\end{proof}
	By Lemma~\ref{Iso-rep&groups-2}, $\mathbb{Z}_p\rtimes_{\gamma_{1,-1}}\operatorname{D}_{2r}\cong \mathbb{Z}_p\rtimes_{\gamma_{-1,1}}\operatorname{D}_{2r}$.
	Then $\gamma_{1,-1}$ and $\gamma_{-1,1}$ correspond to the same rotary augmented PX map.
	We therefore have the following corollary directly.
	\begin{corollary}\label{coro:trivial-PX-map}
		There is a unique rotary map whose underlying graph is $\C^*(p,r,0,-1)$.
		Moreover, this map is $(\ZZ_p\rtimes_{\gamma_{-1,-1}} \D_{2r})$-rotary.
		
		The number of rotary maps whose underlying graph is $\C^*(p,r,0,1)$ is $1$  if $r$ is odd and $2$ if $r$ is even.
		Moreover, these maps are either $(\ZZ_p\rtimes_{\gamma_{1,1}}\D_{2r})$-rotary or $(\ZZ_p\rtimes_{\gamma_{1,-1}}\D_{2r})$-rotary.
	\end{corollary}
	
	Now we give Table~\ref{tab:table}, which summarizes key results from Lemma~\ref{lem-|RPairs|=|IrrReps|}.
	In this table, $t$ is the number of $G$-rotary augmented maps, $\Gamma$ is the underlying graph of the map and $d$ is the degree of $\psi$.
	
	\begin{table}[ht]
		\centering
		\begin{tabular}{ccccc}
			\toprule
			$r$ & $\operatorname{Irr}(D_{2r})$ &  $G$ & $\Gamma$ & t\\
			\midrule
			$r$ is odd&$\gamma_{1,1}$ & $\mathbb{Z}_p\times\operatorname{D}_{2r}$ & $\operatorname{C}^*(p,r,0)$& 1\\
			&$\gamma_{-1,-1}$&$\D_{2pr}$&$\operatorname{C}^*(p,r,0,-1)$&1\\
			&$\psi$ ($d\ge 2$)&$\mathbb{Z}_p^d\rtimes_{\psi}\operatorname{D}_{2r}$&$\operatorname{C}(p,r,d-1)$&$|\psi^{\Aut(\operatorname{D}_{2r})}|$ \\
			\midrule
			$r$ is even&$\gamma_{1,1}$ & $\mathbb{Z}_p\times\operatorname{D}_{2r}$ & $\operatorname{C}^*(p,r,0)$& 1\\
			&$\gamma_{1,-1},\gamma_{-1,1}$&$\mathbb{Z}_p\rtimes_{\gamma_{1,-1}}\operatorname{D}_{2r}$& $\operatorname{C}^*(p,r,0)$& 1\\
			&$\gamma_{-1,-1}$&$\D_{2pr}$&$\operatorname{C}^*(p,r,0,-1)$&1\\
			&$\psi$ ($d\ge 2$)&$\mathbb{Z}_p^d\rtimes_{\psi}\operatorname{D}_{2r}$&$\operatorname{C}(p,r,d-1)$&$|\psi^{\Aut(\operatorname{D}_{2r})}|$ \\
			\bottomrule
		\end{tabular}
		\caption{Irreducible rotary augmented PX maps}
		\label{tab:table}
	\end{table}

	By Lemma~\ref{Iso-rep&groups-2}, there is a one-to-one correspondence between isomorphism classes of groups of the form $\mathbb{Z}_p^d\rtimes_\psi\D_{2r}$ where
	$\psi$ is irreducible and the $\Aut(\D_{2r})$-orbits of $\operatorname{Irr}(\D_{2r})$.
	Then Lemma~\ref{lem-|RPairs|=|IrrReps|} implies the following result.
	
	\begin{theorem}\label{|RPairs|=|IrrReps|}
		The number of irreducible rotary augmented PX map of length $r$ is
		\begin{enumerate}[{\rm (i)}]
			\item  $|\mathrm{Irr}(D_{2r})|$ if $r$ is odd;\vskip0.1in
			
			\item $|\mathrm{Irr}(D_{2r})|-1$ if $r$ is even.
		\end{enumerate}
	\end{theorem}


	Recall that for an involution $x$ of the group $G=V\rtimes_\psi D=\ZZ_p^d\rtimes_\psi\D_{2r}$,  $\rho_x$ denotes an arbitrary element in the set $\{ vx\ |\ 1\neq v\in \C_V(x) \}$. 
	Denote the set of all $\RM(G,\rho_x,y)$ by $\mathcal{RM}_G$.
	Next we show that Construction~\ref{Cons:PX-map} gives all rotary PX maps with respect to $G$, where $G=\ZZ_p^d\rtimes_\psi \D_{2r}$ with $d\geqslant 2$.
	
	\begin{proposition}\label{prop:bijection-ggmap-reporbits}
		Let $x_1,y_1,x_2,y_2\in D$ be involutions such that $D=\langle x_1,y_1\rangle=\langle x_2,y_2\rangle$.
		Let $\sigma\in\Aut(D)$ be given by $\sigma(x_1)=x_2$ and $\sigma(y_1)=y_2$.
		Then
		\begin{enumerate}[{\rm (i)}]
			\item  $\RM(G,\rho_{x_1},y_1)\cong\RM(G,\rho_{x_2},y_2)$
			if and only if $\psi\cong\psi\circ \sigma$; in particular, $|\mathcal{RM}_G|=|\psi^{\Aut(D)}|$; \vskip0.1in
   
			\item  the mapping \[\begin{aligned}
				\psi^{\Aut(\D_{2r})}&\rightarrow \mathcal{RM}_G,\\
				\psi\circ\eta &\mapsto \RM(G,\rho_{\eta(x_1)},\eta(y_1))
			\end{aligned}\]
			is a one-to-one correspondence.
		\end{enumerate}
	\end{proposition}
	\begin{proof}
			%
			%
			%
		Suppose that these two rotary maps are isomorphic.
		Then there is a $f\in \Aut(G)$ such that $f(v_1x_1)=v_2x_2$ and $f(y_1)=y_2$.
		Therefore, \[f(x_1^p)=f((v_1x_1)^p)=(v_2x_2)^p=x_2^p.\]
		Since $|x_1|=|x_2|=2$, we have that $f(x_1)=x_2$
		which implies that $f(D)=D$ and $f|_D=\sigma$.
		By Lemma~\ref{Aut(D)_phi}, $f|_D\in\Aut(D)_\psi$
		which means that $\psi\cong\psi\circ\sigma$.
		Now suppose that $\psi\cong\psi\circ\sigma$.
		Denote by $U\cong\mathbb{Z}_p^d$ the $D$-module induced by $\psi\circ\sigma$.
		Then there exists $D$-module isomorphism $f_0:V\rightarrow U$.
		Define $f:G\rightarrow G$ by \[vb\mapsto f_0(v)\sigma(b).\]
		One can check that $f$ is an automorphism.
		Moreover, $f(x_1)=x_2$ and $f(y_1)=y_2$.
		Note that $1\ne f(v_1)\in\operatorname{C}_V(x_2)$.
		Hence $f(v_1x_1)=f(v_1)x_2$.
		We have that $\RM(G,v_1x_1,y_1)\cong\RM(G,f(v_1)x_2,y_2)$.
		By Proposition~\ref{prop:equivalence-ggMap}, $\RM(G,\rho_{x_1},y_1)\cong\RM(G,\rho_{x_2},y_2)$. 
		Therefore, $|\mathcal{RM}_G|=|\Aut(G)|/|\Aut(G)_\psi|=|\psi^{\Aut(\D_{2r})}|$ and
		this completes the proof.
	\end{proof}

	\begin{corollary}\label{coro:isomorphic-ggMap}
		Each $G$-rotary PX map lies in $\mathcal{RM}_G$, and		
		the mapping in Proposition~\ref{prop:equivalence-ggMap} is a one-to-one correspondence between $G$-rotary PX maps and $\psi^{\Aut(\D_{2r})}$.
		
		%
	\end{corollary}
	\begin{proof}
		Recall that $\RM(G,\rho_x,y)$ such that $x$ and $y$ are two involutions generating $D$ equals
		\[|\Aut(\operatorname{D}_{2r})|/|\Aut(\operatorname{D}_{2r})_\psi|=|\psi^{\Aut(\operatorname{D}_{2r})}|.\]
  This count matches the number of isomorphism classes of rotary maps $\RM(G,\rho,\tau)$ where $|\rho|=2p$, as given by Theorem~\ref{|RPairs|=|IrrReps|}.
		This completes the proof.
		%
	\end{proof}

	\begin{remark}
		Corollary \ref{coro:isomorphic-ggMap} actually gives a one-to-one correspondence between irreducible rotary PX maps of length $r$ and irreducible representations of $\D_{2r}$ of degree greater than $1$.
		Together with Corollary \ref{coro:trivial-PX-map}, we have that
		there is a one-to-one correspondence between irreducible rotary augmented PX maps of length $r$ and irreducible representations of $\D_{2r}$ except $\gamma_{-1,1}$.
	\end{remark}

	\section{The decomposition of rotary PX maps}\label{sec}

	In this section, we establish that every rotary augmented PX map decomposes as a direct product of irreducible rotary augmented PX maps, and that this decomposition is unique up to isomorphism. 
	Recall that we always assume \( p \) is an odd prime, \( r \geq 3 \), and \( p \nmid r \).
	Thus, the group algebra $\mathbb{F}_p\D_{2r}$ is semisimple.
	Applying Maschke's theorem, we prove the existence of such decomposition in Theorem~\ref{thm:existence-decomposition}.
	Lemma~\ref{lem:unqiue-minimal-normal-subgroup} plays a key role in proving the uniqueness of the decomposition.

	We first study quotient maps of rotary PX maps. Recall the definition of augmented PX graphs (see Definition~\ref{def:augPXgraph}).

	\begin{lemma}\label{lem:G-AT-QuoMap}
		Let $\mathcal{M}$ be a $G$-rotary map whose underlying graph is $\operatorname{C}(p,r,s)$.
		Suppose that $V$ is a normal $p$-subgroup of $G$ satisfying $|V|<p^{s+1}$.
		Then $\mathcal{M}/V$ is a $G/V$-rotary augmented PX map of length $r$. 
	\end{lemma}
	
	\begin{proof}
		Let $(\alpha,e,\beta)$ be an arc of $\mathcal{M}$.
		The valency of $\operatorname{C}(p,r,s)$ is $2p$ and since $G$ is arc-regular, it follows that $|G_\alpha|=2p$.
		As $p$ is odd, $G_\alpha\nleq V$ and $G_e\nleq V$.
		So the quotient $\mathcal{M}/V$ is well defined.	
		Let $|V|=p^t$ for some integer $t$ with $0 \le t < s+1$.
		Then $G/V\cong\mathbb{Z}_p^{s-t+1}\rtimes_\psi\operatorname{D}_{2r}$.
		Observe that $|G_\alpha\cap V|\in\{1, p\}$.
		
		Let $\bar G_\a=G_\a/G_\a\cap V$.
		If $|G_\alpha\cap V|=1$, then $|\bar{G}_\alpha|=2p$. 
		By Lemma~\ref{lem:arc-regular-PX-map}, the underlying graph of $\mathcal{M}/V$ is $\operatorname{C}^*(p,r,s-t,\delta)$. 
		If $|G_\alpha\cap V|=p$, then $|\bar{G}_\alpha|=2$ and $G/V$ is dihedral.
		Hence $G/V\cong\operatorname{D}_{2pr}$, and $\mathcal{M}/V$ has underlying graph $\operatorname{C}^*(p,r,0,-1)$.
	\end{proof}
	
	Now we prove the existence of the decomposition of rotary PX maps.

	\begin{theorem}\label{thm:existence-decomposition}
		Let $\mathcal{M}$ be a rotary augmented PX map of length $r$.
		Then $\mathcal{M}$ is isomorphic to a direct product of irreducible rotary augmented PX maps of length $r$.
		
		Conversely, every direct product of irreducible rotary augmented PX maps of length $r$
		is a rotary augmented PX map whose underlying graph is of length $r$.
	\end{theorem}
	\begin{proof}
		If the underlying graph of $\mathcal{M}$ is not a PX graph, then
		by Theorem~\ref{thm:arc-regular-PX-graph}, we have that $\mathcal{M}$ is an irreducible rotary augmented PX map. 
		This completes the case.
		
		Now suppose that $\mathcal{M}$ is a rotary PX map.
		Let $G$ be the orientation preserving automorphism group of $\calM$.
		Then $G$ is arc-regular and $\calM$ is $G$-rotary.
		By Theorem~\ref{thm:arc-regular-PX-graph}, $G=\mathbb{Z}_p^{s+1}\rtimes_{\psi}\operatorname{D}_{2r}$.
		Identify $\mathbb{Z}_p^{s+1}$ with $\mathbb{F}_p^{s+1}$.
		Note that the odd prime $p$ is coprime to $r$.
		By Maschke's theorem, there are minimal normal subgroups $V_1,\ldots,V_n\le\mathbb{Z}_p^{s+1}$ of $G$ such that $\mathbb{Z}_p^{s+1}=V_1\times\cdots\times V_n$.
		For every $i\in\{1,\ldots,n\}$, let $U_i=\prod_{j\ne i}V_j$. 
		Then $U_i$ is a normal subgroup of $G$ and $\bigcap_i U_i=1$.
		Then by Lemma~\ref{lem:G-AT-QuoMap}, $\mathcal{M}/U_i$ is a $G/U_i$-rotary augmented PX map of length $r$.
		Since $G/U_i\cong V_i\rtimes_{\psi|_{V_i}}\operatorname{D}_{2r}$,
		we have that $\mathcal{M}/U_i$ is irreducible.
		By Lemma~\ref{lem:direct-product-rotary-map}, $\mathcal{M}\cong\prod_i\mathcal{M}/U_i$.

		Now we study the other direction.
		For every $i\in\{1,\ldots,n\}$, let \[\mathcal{M}_i=\RM(G_i,\rho_i,\tau_i),\]
		where $G_i=\mathbb{Z}_p^{s_i}\rtimes\mathrm{D}_{2r}$.
		Let 
		$\rho=(\rho_1,\ldots,\rho_n),\ \tau=(\tau_1,\ldots,\tau_n)\in \prod_{i\in\{1,\ldots,n\}} G_i$, let $H=\langle \rho , \tau \rangle$ and let $\mathcal{M}=\prod_{i\in\{1,\ldots,n\}}\mathcal{M}_i$.
		Then $\mathcal{M}=\RM(H,\rho , \tau)$.
		By our definition, either $|\rho|=2$ or $|\rho|=2p$.
		If $|\rho|=2$, it is clear that $\mathcal{M}\cong\mathcal{M}_1\cong\cdots\cong\mathcal{M}_n$ whose underlying graph is $\operatorname{C}_{pr}$.
		
		Now suppose that $|\rho|=2p$.
		We claim that $H\cong \mathbb{Z}_p^d\rtimes\mathrm{D}_{2r}$ for some $1\le d\le s_1+\cdots+s_n$.
		For every $i\in\{1,\ldots,n\}$, let $\pi_i$ be the natural projection from $G_i$ to $\operatorname{D}_{2r}$, that is,  $\pi_i((x,y))=y$ for every $(x,y)\in G_i$.
		Let $\pi:\prod_{i\in\{1,\ldots,n\}}G_i\rightarrow \operatorname{D}_{2r}^n$ be a homomorphism given by $\pi(g_1,\ldots,g_n)=(\pi_1(g_1),\ldots,\pi_n(g_n))$.
		Let $K=\ker(\pi)$. Then 
		\[K= \prod_{i\in\{1,\ldots,n\}}\mathbb{Z}_p^{s_i}\trianglelefteq \prod_{i\in\{1,\ldots,n\}}G_i\]
		is the unique Sylow $p$-subgroup of $\prod_{i\in\{1,\ldots,n\}}G_i$.
		Since $\rho$ is of order $2p$ and $\tau$ is of order $2$, we have that both $\pi(\rho)$ and $\pi(\tau)$ are of order $2$.
		Therefore, $\pi(H)=\langle\pi(\rho),\pi(\tau)\rangle$ is dihedral.
		Since for every $i\in\{1,\ldots,n\}$, $G_i=\langle \rho_i,\tau_i\rangle$,
		we have that $\langle\pi_i(\rho_i),\pi_i(\tau_i)\rangle=\operatorname{D}_{2r}$, which implies that $\pi_i(\rho_i)\pi_i(\tau_i)$ is of order $r$ for every $i\in\{1,\ldots,n\}$.
		Hence $\pi(\rho)\pi(\tau)=(\pi_1(\rho_1)\pi_1(\tau_1),\ldots,\pi_n(\rho_n)\pi_n(\tau_n))$ is of order $r$. 
		This implies that $\pi(H)\cong\operatorname{D}_{2r}$.
		By the homomorphism theorem, $\pi(H)\cong H/(H\cap K)$.
		Write $H\cap K\cong \mathbb{Z}_p^d$. 
		Then by the Schur-Zassenhaus theorem, $H\cong\mathbb{Z}_p^d\rtimes\operatorname{D}_{2r}$.
		
		Let $\Gamma$ be the underlying graph of $\mathcal{M}$. 
		Since $\langle\rho\rangle$ is the vertex stabilizer and $|\rho|=2p$, it follows that $\mathbb{Z}_p^{d}\cap \langle\rho\rangle=\mathbb{Z}_p$. 
		Therefore, $\mathbb{Z}_p^{d}$ is not semi-regular on the vertex set of $\Gamma$.
		By Lemma~\ref{lem:arc-regular-PX-map}, $\Gamma=\operatorname{C}^{*}(p,r,d-1)$, which completes the proof.
	\end{proof}
	
	Now we look into the uniqueness of the decomposition. Before that, we need to establish Lemma~\ref{lem:quotient-map} and Lemma~\ref{lem:unqiue-minimal-normal-subgroup}.
	
	\begin{lemma}\label{lem:quotient-map}
		Let $\mathcal{M}$ be a $G$-rotary map, and let $V_1,V_2$ be normal subgroups of $G$. 
		Then $\mathcal{M}/V_1\cong\mathcal{M}/V_2$ if and only if $V_1=V_2$.
	\end{lemma}
	\begin{proof}
		If $V_1=V_2$, then it is clear that $\calM/V_1\cong \calM/V_2$.
		
		Write $\mathcal{M}=\RM(G,\rho,\tau)$.
		If $\mathcal{M}/V_1\cong\mathcal{M}/V_2$, then there is an isomorphism $\sigma: G/V_1 \to G/V_2$ such that $\sigma(\rho V_1)=\rho V_2$ and $\sigma(\tau V_1)=\tau V_2$.
		Let $\pi_{V_1}:G\rightarrow G/V_1$ and $\pi_{V_2}:G\rightarrow G/V_2$ be natural projections from $G$ onto $G/V_1$ and $G/V_2$, respectively.
		Then $\pi_{V_2}(\rho)=\sigma\circ\pi_{V_1}(\rho)$ and $\pi_{V_2}(\tau)=\sigma\circ\pi_{V_1}(\tau)$.
		Since $G=\langle \rho , \tau \rangle$, we have that $\pi_{V_2}=\sigma\circ\pi_{V_1}$.
		Since $\ker(\pi_{V_2})=V_2$ and $\ker(\sigma\circ\pi_{V_1})=V_1$, we conclude $V_1=V_2$.
		This completes the proof.
	\end{proof}

	\begin{lemma}\label{lem:unqiue-minimal-normal-subgroup}
		Let $V_1=\mathbb{Z}_p^{d_1}$ and let $V_2=\mathbb{Z}_p^{d_2}$.
		Let $G=(V_1\times V_2)\rtimes_{(\psi_1,\psi_2)} \operatorname{D}_{2r}$
		where $\psi_i:\operatorname{D}_{2r}\rightarrow \operatorname{GL}(V_i)$ is irreducible for every $i\in\{1,2\}$.
		If $G$ has a rotary pair $(\rho,\tau)$ such that $|\rho|=2p$ and $|\tau|=2$,
		then $\psi_1\not\cong\psi_2$.
	\end{lemma}
	\begin{proof}
		Assume that $\psi_1\cong\psi_2$. We prove that this leads to a contradiction.
		Note that $p\nmid r$.
		By the Schur-Zassenhaus theorem, all subgroups of order $2r$ are conjugate in $G$.
		Without loss of generality, we can assume that $\tau\in \operatorname{D}_{2r}$.
		Write $\rho=v\tau_1$, where $v\in V_1\times V_2$ and $\tau_1\in\operatorname{D}_{2r}$.
		Then $\operatorname{D}_{2r}=\langle\tau,\tau_1\rangle$.
		Let $v'=\rho^2$ and let $\tau'=\rho^p$.
		Then $\langle \rho\rangle=\langle v'\rangle\times\langle \tau'\rangle$.
		Let $D=\langle \tau',\tau\rangle$.
		Then $D$ is a dihedral group. 
		Since $G=\langle v',\tau',\tau\rangle$,
		we have that $G=(V_1\times V_2)D$.
		Therefore, either $D\cong\operatorname{D}_{2r}\mbox{ or }\operatorname{D}_{2pr}$.
		
		In the latter case, $D\cap (V_1\times V_2)$ is of order $p$ and normal in $D$.
		Treat $V_1\times V_2$ as a $D_{2r}$-module.
		Then $D\cap (V_1\times V_2)$ is a simple $D_{2r}$-submodule.
		By Maschke's theorem, $V_1\cong V_2\cong\mathbb{Z}_p$.
		By Lemma~\ref{lem:irreducible-representation-dihedral} {\rm (i)} and {\rm (ii)},
		$\psi_1\in\{\gamma_{1,1},\gamma_{-1,-1},\gamma_{-1,1},\gamma_{1,-1}\}$ (if possible).
		If $\psi_1$ is trivial, then $G\cong\mathbb{Z}_p^2\times\operatorname{D}_{2r}$.
		This leads to a contradiction: $\mathbb{Z}_p^2$ is isomorphic to the cyclic group $G/\operatorname{D}_{2r} = \langle \rho \operatorname{D}_{2r}\rangle$.
		If $\psi_1=\gamma_{-1,-1}$, then $G$ does not contain any element of order $2p$, which is also a contradiction.
		If $\psi_1\in\{\gamma_{1,-1},\gamma_{-1,1}\}$, then $v'=vv^{\tau_1}\ne 1$.
		Therefore, $v^{\tau_1}\ne v^{-1}$, which
		implies that $\tau_1$ centralizes $V_1\times V_2$.
		Hence $\tau'=\tau_1$ and $D=\operatorname{D}_{2r}\not\cong\operatorname{D}_{2pr}$.
		
		We therefore deduce that $D\cong\operatorname{D}_{2r}$ and $G=(V_1\times V_2)\rtimes_{(\eta_1,\eta_2)} D$.
		Since $D$ is conjugate to $\operatorname{D}_{2r}$, then there is an isomorphism $\sigma:D\rightarrow \operatorname{D}_{2r}$
		such that $\eta_1=\psi_1\circ \sigma$ and $\eta_2=\psi_2\circ\sigma$.
		Then $\eta_1\cong\eta_2$, and denote by $f$ the $D$-module isomorphism from $V_1$ to $V_2$.
		Since $v'\in V_1\times V_2$,
		we have that $\langle (v')^D\rangle\le V_1\times V_2$.
		Also since $G=\langle v',D\rangle$, we have that $G=\langle (v')^D\rangle D$.
		Therefore, $\langle (v')^D\rangle=V_1\times V_2$.
		Let $v'=(x,y)$ where $x\in V_1$ and $y\in V_2$. 
		Note that $v'\in\operatorname{C}_{V_1\times V_2}(\tau')$.
		We have that $x\in\operatorname{C}_{V_1}(\tau')$ and $y\in\operatorname{C}_{V_2}(\tau')$.
		Then there exists $y'\in V_1$ such that $f^{-1}(y)=y'$.
		By Lemma~\ref{lem:module-homomorphism} {\rm (ii)}, there exists $\zeta'\in \mathrm{End}_D(V_1)$ such that 
		\[\zeta'(x)=y'=f^{-1}(y).\]
		Define $\zeta=f\circ \zeta'$.
		
		Then  $\zeta:V_1\rightarrow V_2$ is a $D$-module isomorphism such that  $\zeta(x)=y$.
		
		Therefore, $v'\in\{(w,\zeta(w))\mid w\in V_1\}$.
		Note that $\{(w,\zeta(w))\mid w\in V_1\}\cong \ZZ_p^{d_1}$ is normal in $G$. 
		We have that $\langle (v')^D\rangle\le\{(w,\zeta(w))\mid w\in V_1\}$.
		However, this is impossible because $\{(w,\zeta(w))\mid w\in V_1\}\ne V_1\times V_2$.
		Therefore, we have that $\psi_1\not\cong\psi_2$.
	\end{proof}

	\begin{corollary}\label{cor:unique-minimal-normal-subgroup}
		For every $i\in\{1,\ldots,n\}$, let $V_i=\mathbb{Z}_p^{d_i}$.
		Let $G=(V_1\times \cdots\times V_n)\rtimes_{(\psi_1,\ldots,\psi_n)} \operatorname{D}_{2r}$
		where $\psi_i:\operatorname{D}_{2r}\rightarrow \operatorname{GL}(V_i)$ is irreducible for every $i\in\{1,\ldots,n\}$.
		If $G$ has a rotary pair $(\rho,\tau)$ such that $|\rho|=2p$ and $|\tau|=2$,
		then $\psi_i\not\cong\psi_j$ for every $1\le i<j\le n$.
	\end{corollary}
	\begin{proof}
		Let $V_{i,j}=\prod_{k\ne i,k\ne j}V_k$.
		Then $G/V_{ij}\cong (V_i\times V_j)\rtimes_{(\psi_i,\psi_j)}\operatorname{D}_{2r}$.
		Since $G$ is generated by $\rho$ and $\tau$, we have that $G/V_{i,j}$
		is generated by $\rho V_{ij}$ and $\tau V_{ij}$.
		If $|\rho V_{ij}|=2$, then $G/V_{ij}$ is a dihedral group, which implies that the Sylow $p$-subgroup $V_1\times V_2$ of $G/V_{ij}$ is cyclic. 
		However, this is impossible.
		Therefore, $|\rho V_{ij}|=2p$. By Lemma~\ref{lem:unqiue-minimal-normal-subgroup}, we have that $\psi_i\not\cong\psi_j$.
		Since $i,j$ are arbitrary, we have that $\psi_i\not\cong\psi_j$ for every $1\le i<j\le n$.
	\end{proof}
	
	\begin{corollary}\label{coro:auto-ProducMap}
		Let $\mathcal{M}_1,\ldots,\mathcal{M}_m$ be irreducible $\mathbb{Z}_p^{d_i}\rtimes\operatorname{D}_{2r}$-rotary augmented PX maps. If $\mathcal{M}_{i}\not\cong\mathcal{M}_j$ for $i\ne j$,
		then the orientation preserving automorphism group $G$ of $\mathcal{M}_1\times\cdots\times\mathcal{M}_m$ is $ (\mathbb{Z}_p^{d_1}\times\cdots\times\mathbb{Z}_p^{d_m})\rtimes \operatorname{D}_{2r}$.
	\end{corollary}
	\begin{proof}
		If $m=1$, it is clear that $\mathcal{N}\cong \mathcal{M}_1$. 
		
		Now we suppose that $m\ge 2$.
		Let $G_i=\mathbb{Z}_p^{d_1}\rtimes_{\psi_i}\operatorname{D}_{2r}$ and let $G$ be the orientation preserving automorphism group of $\calM=\mathcal{M}_1\times\cdots\times\mathcal{M}_m$.
		Write $\mathcal{M}_i=\RM(\mathbb{Z}_p^{d_1}\rtimes_{\psi_i}\operatorname{D}_{2r},\rho_i,\tau_i)$ and
		then \[G=\langle (\rho_1,\ldots,\rho_m),(\tau_1,\ldots,\tau_m)\rangle.\]
		By Theorem~\ref{thm:existence-decomposition} and Corollary~\ref{cor:aut-map},
		$G\cong\mathbb{Z}_p^{d}\rtimes\operatorname{D}_{2r}$ for some integer $d$.
		Let $V_1,\ldots,V_n$ be all minimal normal subgroups of $G$ contained in $\mathbb{Z}_p^d$.
		Then by Corollary~\ref{cor:unique-minimal-normal-subgroup},
		$G=(V_1\times\cdots\times V_n)\rtimes_{(\eta_1,\ldots,\eta_n)}\operatorname{D}_{2r}$
		where $\eta_i\not\cong\eta_j$ for every $i,j\in\{1,\ldots,n\}$ with $i\ne j$.
		Let $p_i$ be the natural projection from $G$ to $G_i$.
		Then $\ker(p_i)\le V_1\times\cdots\times V_n$ which implies that
		$\ker(p_i)=V_{i_1}\times\cdots V_{i_k}$ for some $1\le i_1<i_2<\cdots<i_k\le n$.
		Since $G_i\cong G/\ker(p_i)\cong\prod_{i\in\{1,\ldots,n\}{\setminus}\{ i_1,\ldots,i_k\}} V_i\rtimes\operatorname{D}_{2r}$,
		we have that $\prod_{i\in\{1,\ldots,n\}{\setminus}\{i_1,\ldots,i_k\}} V_i$ is minimal normal in $G_i$.
		Therefore, $k=n-1$ and $V_j\cong\mathbb{Z}_p^{d_j}$ where \[j\in\{1,\ldots,n\}{\setminus}\{i_1,\ldots,i_{n-1}\}.\] 
		Since $\mathcal{M}/\ker(p_i)\cong\mathcal{M}_i$ and $\mathcal{M}_i\not\cong\mathcal{M}_j$ when $i\ne j$, by Lemma~\ref{lem:quotient-map}, we have that
		$\ker(p_i)\neq \ker(p_j)$ when $i\ne j$.
		Therefore, $n=m$ and $G\cong (\mathbb{Z}_p^{d_1}\times\cdots\times\mathbb{Z}_p^{d_m})\rtimes\operatorname{D}_{2r}$.
	\end{proof}
	
	\begin{corollary}\label{iso}
		Using notation in Corollary~\ref{coro:auto-ProducMap}.
		Suppose  that an irreducible rotary augmented PX map $\calN$ of length $r$ is isomorphic to a quotient map of $\mathcal{M}_1\times\cdots\times\mathcal{M}_m$.
		Then $\calN\cong \calM_i$ for some $i\in \{1,\ldots,m\}$.

		%
	\end{corollary}
	\begin{proof}
		Let $L$ be the orientation preserving automorphism group of $\calN$.
		Since $L\cong\mathbb{Z}_p^{d'}\rtimes\operatorname{D}_{2r}$ for some ${d'}$, there exists a subgroup $V\le V_1\times\cdots\times V_m$ such that $\mathcal{M}/V\cong\mathcal{N}$.
		Since $\mathcal{N}$ is irreducible, we have that $V=V_{i_1}\times\cdots\times V_{i_{m-1}}$ for some $1\le i_1<\cdots<i_{m-1}\le m$.
		Hence $V=\ker(p_j)$ where $j\in\{1,\ldots,m\}$ is the integer not equal to any of $i_1,\ldots,i_{m-1}$.
		This implies that $\mathcal{N}\cong\mathcal{M}_j$.
	\end{proof}
	
	Note that for every arc-regular map $\mathcal{N}$, we have $\mathcal{N}\cong\mathcal{N}\times\mathcal{N}$.
	Therefore, if there exists a decomposition $\mathcal{M}\cong \mathcal{N}_1\times\cdots \times \mathcal{N}_n$ for an arc-regular map $\mathcal{M}$,
	then we can pick $i_1,i_2,\ldots,i_k$ such that $\mathcal{M}\cong\mathcal{N}_{i_1}\times\cdots\times\mathcal{N}_{i_k}$ and $\mathcal{N}_{i_{j_1}}\not\cong\mathcal{N}_{i_{j_2}}$ if $j_1\neq j_2$.
	Now we have the following result.
	
	\begin{theorem}\label{ud}
		Let $\mathcal{M}_1,\ldots,\mathcal{M}_m,\mathcal{N}_1,\ldots,\mathcal{N}_n$ be irreducible rotary augmented PX maps
		whose underlying graph is of length $r$.
		Let $\mathcal{M}_{i}\not\cong\mathcal{M}_j$ for $i\ne j$ and let $\mathcal{N}_i\not\cong\mathcal{N}_j$ for $i\ne j$.
		Then $\mathcal{M}_1\times\cdots\times\mathcal{M}_m\cong\mathcal{N}_1\times\cdots\times\mathcal{N}_n$ if and only if $m=n$ and there is a permutation $\sigma$ of $\{1,\ldots,n\}$ such that $\mathcal{M}_i\cong\mathcal{N}_{\sigma(i)}$ for all $i$.
	\end{theorem}
	\begin{proof}
		Let $\mathcal{M}=\mathcal{M}_1\times\cdots\times\mathcal{M}_m$.
		Since $\mathcal{M}_1\times\cdots\times\mathcal{M}_m\cong\mathcal{N}_1\times\cdots\times\mathcal{N}_n$, we have that $\mathcal{N}_i$ is isomorphic to a quotient map of $\mathcal{M}$. By Corollary~\ref{iso}, for every $i\in\{1,\ldots,n\}$,
		there exists $j\in\{1,\ldots,m\}$ such that $\mathcal{N}_i\cong \mathcal{M}_j$. We thereby conclude that $n\leq m$ because $\mathcal{N}_i$ are pairwise non-isomorphic.
		
		In a similar way, we obtain that $m\leq n$ and $\mathcal{M}_i\cong \mathcal{N}_j$ for all $1\leq i\leq m$. Therefore, $n=m$ and there is a permutation $\sigma$ of $\{1,\ldots,n\}$ such that $\mathcal{M}_i\cong\mathcal{N}_{\sigma(i)}$ for all $i$.
	\end{proof}
	Theorem~\ref{fenjie} is a direct consequence of combining Theorems~\ref{thm:existence-decomposition}, \ref{ud}, and Corollary~\ref{iso}.
	\begin{theorem}\label{fenjie}
		Let $\mathcal{M}$ be a $G$-rotary augmented PX map of length $r$.
		Then 
		$$\mathcal{M}=\mathcal{M}_1\times\dots\times \mathcal{M}_k,$$
		where $\mathcal{M}_i$ is an irreducible $G_i$-rotary augmented PX map of length $r$, $\mathcal{M}_i\ncong\mathcal{M}_j$ for $1\leq i\neq j\leq k$. Moreover, this decomposition is unique up to isomorphism and $$G\cong (\mathbb{Z}_p^{d_1}\times\cdots\times\mathbb{Z}_p^{d_k})\rtimes \operatorname{D}_{2r}$$ where $G_i=\mathbb{Z}_p^{d_i}\rtimes \operatorname{D}_{2r}$.
	\end{theorem}

	\section{Reducible PX maps}
	
	We now prove Theorem~\ref{thm:main}.
	\begin{proof}[Proof of Theorem~\ref{thm:main}]
		Let $\mathcal{M}$ be a $G$-rotary PX map with underlying graph $\C(p,r,s)$. 
		By Theorem~\ref{fenjie}, $\mathcal{M}=\mathcal{M}_1\times\dots\times \mathcal{M}_k$ where $\mathcal{M}_i$ is an irreducible rotary augmented PX maps and $\calM_i\not\cong \calM_j$ if $i\neq j$.
		
		By Corollary~\ref{coro:trivial-PX-map} and~\ref{coro:isomorphic-ggMap}, there is a one-to-one correspondence between irreducible rotary augmented PX maps of length $r$ and irreducible representations of $\D_{2r}$ (excluding $\gamma_{-1,1}$).
		For every $i$, let $\psi_i$ be the corresponding irreducible representation and let $d_i$ be the degree of $\psi_i$.
            Then, for each $i$, $\calM_i$ is $(\mathbb{Z}_p^{d_i}\rtimes_{\psi_i}\D_{2r})$-rotary map.		
		Then \[G=(\ZZ_p^{d_1}\times \cdots \times \ZZ_p^{d_k})\rtimes_{\sum \psi_i}\D_{2r}\] by Theorem~\ref{fenjie}.
		It follows from Theorem~\ref{thm:arc-regular-PX-graph} that $s+1=\sum d_i$.
		By Corollary~\ref{cor:unique-minimal-normal-subgroup}, $\sum \psi_i$ is a multiplicity-free representation of $\D_{2r}$.
		The other direction is clear.
            This completes the proof.
	\end{proof}

	\begin{proof}[Proof of Corollary~\ref{cor:faithful}]
		By Lemma~\ref{lem:faith-rep}, if $d$ is odd or $d$ is even with $p^{d/2}\not\equiv -1\pmod{r}$, then the set of degrees of irreducible representations of $\D_{2r}$ is \[\{1,\ 1,\ 2d,\dots,2d\}.\]
		If $d$ is even with $p^{d/2}\equiv -1\pmod{r}$, then the set of degrees of irreducible representations is \[\{1,\ 1,\ d,\dots,d\}.\] Then it follows from Theorem~\ref{thm:main}. 
	\end{proof}
	
	\begin{example}
		Let $p=13$ , let $r=7$ and let $s\in\{1,\ldots,6\}$.
		Then $d=2$ is the least positive integer such that $r\mid p^d-1$.
		Note that $d$ is even and $p^{d/2}\equiv -1\pmod{r}$.
		Then by Lemma~\ref{lem:faith-rep} {\rm (ii)}, there are $\varphi(r)/d=3$ distinct faithful irreducible representations and these representations are all of degree $d$.
		Since $7$ is a prime, all irreducible representations of $\operatorname{D}_{14}$ of degree greater than $1$ is faithful.
		Therefore, $\operatorname{D}_{14}$ has five irreducible representations of degree $1,1,2,2,2$ respectively.
		By Corollary~\ref{cor:faithful}, there is a rotary map with underlying graph isomorphic to $\operatorname{C}(p,r,s)$.
		Write $s+1=2k+m$ where $0\le m\le 1$.
		If $s$ is odd, then $m=0$.
          By Theorem~\ref{thm:main}, to achieve a sum of degrees equal to $s+1=2k$, we must select an even number of the  representations of degree $1$. 
          This means we select either both of them or neither of them.
          Therefore, there are $\binom{3}{k-1}+\binom{3}{k}$ rotary PX maps with underlying graph isomorphic to $\operatorname{C}(p,r,s)$.
		If $s$ is even, then $m=1$.
		By Theorem~\ref{thm:main}, we select one of the representations of degree $1$.
		Therefore, there are $2\binom{3}{k}$
		rotary PX maps with underlying graph isomorphic to $\operatorname{C}(p,r,s)$.
	\end{example}
	
	%
	%
	%
	%

	\appendix
	\section{Representation theory of dihedral groups}\label{sec:appendix}
	\begin{proof}[Proof of Lemma~\ref{lem:irreducible-representation-dihedral}]
		Since $\mathbb{F}_p^\times$ is cyclic, we have that $D'\le\ker(\psi)$.
		If $r$ is odd, then $D'\cong\mathbb{Z}_r$ and $D/D'\cong\mathbb{Z}_2$.
		Therefore, $\psi\in\{\gamma_{1,1},\gamma_{-1,-1}\}$.
		If $r$ is even, then $D'\cong\mathbb{Z}_{r/2}$ and $D/D'\cong\mathbb{Z}_2^2$.
		Therefore, we have $\psi\in\{\gamma_{1,1},\gamma_{-1,-1},\gamma_{1,-1},\gamma_{-1,1}\}$.
		This completes the proof of (i) and (ii).
		Now we prove (iii).
		Since $z\notin \operatorname{Z}(D)$, there exists an involution $z_1\in D$ such that $D=\langle z,z_1\rangle$.
		Suppose $v\in E_1(z)\cap E_{1}(z_1)$.
		If $v\ne 0$, then $\mathbb{F}_pv$ is invariant under the action of $D$.
		Since $D$ is irreducible on $V$, we have that $V=\mathbb{F}_pv$ of dimension $1$.
		This contradicts the assumption that $d\ge 2$.
		Therefore, $E_1(z)\cap E_{1}(z_1)=0$.
		Similarly, $E_1(z)\cap E_{-1}(z_1)=0$, $E_{-1}(z)\cap E_{1}(z_1)=0$ and
		$E_{-1}(z)\cap E_{-1}(z_1)=0$.
		Let $d_1=\dim(E_1(z))$ and let $d_{-1}=\dim(E_{-1}(z))$.
		Note that $E_1(z)\oplus E_{-1}(z)=V$.
		Therefore, we have that $\dim(E_1(z_1))\le\min(d_1,d_2)$
		and $\dim(E_{-1}(z_1))\le\min(d_1,d_2)$.
		Also note that $E_1(z_1)\oplus E_{-1}(z_1)=V$.
		We have that $d=\dim(E_1(z_1))+\dim(E_{-1}(z_1))\le 2\min(d_1,d_2)$.
		Since $d=d_1+d_2$, we have that $d_1=d_2=d/2$.
	\end{proof}
	
	\begin{proof}[Proof of Lemma~\ref{lem:module-homomorphism}]
        Let $C=\langle c\rangle$ and let $M_1=\operatorname{End}_C(M)$.
		For every $f\in M_1$ and integers $i,j$, define an operation $(c^ib^j)\cdot f$ given by $(c^ib^j)\cdot f(m)=c^ib^j(f(b^jm))$ for every $m\in M$. We first show that $M_1$ with this operation forms a $D$-module and $M\cong M_1$.
  
		One can check that $(c^ib^j)\cdot f\in M_1$ for every $f\in M_1$, and integers $i,j$.
		Now let $c^{i_1}b^{j_1},c^{i_2}b^{j_2}\in D$ and let $m\in M$.
		Then \[\begin{aligned}(((c^{i_1}b^{j_1})(c^{i_2}b^{j_2}))\cdot f)(m)&=((c^{i_1+(-1)^{j_1}i_2}b^{j_1+j_2})\cdot f)(m)\\
			&=c^{i_1+(-1)^{j_1}i_2}b^{j_1+j_2}f(b^{j_1+j_2}m).\end{aligned}\]
		We also have
		\[\begin{aligned}
			(c^{i_1}b^{j_1}\cdot(c^{i_2}b^{j_2}\cdot f))(m)&=c^{i_1}b^{j_1}((c^{i_2}b^{j_2}\cdot f)(b^{j_1}m))\\
			&=c^{i_1}b^{j_1}((c^{i_2}b^{j_2})f(b^{j_1+j_2})m)\\
			&=c^{i_1+(-1)^{j_1}i_2}b^{j_1+j_2}f(b^{j_1+j_2}m)
		\end{aligned}\]
		Therefore, $M_1$ with this operation forms a $D$-module.
		
		Let $0\ne m\in E_1(b)$. 
		Let $\psi:M_1\rightarrow M$ be a map given by $f\mapsto f(m)$.
		One can check that $\psi$ is a linear map.
		Let $c^ib^j\in D$ and $f\in M_1$.
		Then \[\psi(c^ib^j\cdot f)=(c^ib^j\cdot f)(m)=c^ib^jf(b^jm)=c^ib^jf(m)=c^ib^j\psi(f).\]
		Therefore, $\psi$ is a $D$-module homomorphism.
		As $\psi({\rm id})=m$, we have that $\psi\ne 0$. 
		Since $M$ is a simple $D$-module,
		we thereby have that $\psi$ is surjective.
		Let $f\in\ker(\psi)$. 
		Then $f(m)=0$. 
		Since $M$ is a simple $D$-module and $m\in E_1(b)$, 
		we have that $M=\mathbb{F}_pDm=\mathbb{F}_pCm$.
		Therefore $f(M)=0$ as $f$ is a $C$-module homomorphism.
		This implies that $\ker(\psi)=0$, hence $\psi$ is a $D$-module isomorphism.

        Now we prove the statements {\rm (i)} and {\rm (ii)}.
		
  {\rm (i)} It can be seen that $\operatorname{End}_D(M)=\operatorname{C}_{M_1}(b)$.
		Then by Lemma~\ref{lem:irreducible-representation-dihedral} {\rm (iii)}, $\dim_{\mathbb{F}_p}(\operatorname{End}_D(M))=d/2$.
		By Wedderburn's little theorem, the endomorphism ring $\operatorname{End}_D(M)$ is a field of characteristic $p$.
		We then have that $\operatorname{End}_D(M)\cong \mathbb{F}_{p^{d/2}}$. 
		Therefore, $(\operatorname{End}_D(M))^\times\cong \mathbb{Z}_{p^{d/2}-1}$.
		
		{\rm (ii)} It is clear that $E_1(b)$ is invariant under $\operatorname{End}_D(M)$.
		Let $\psi:\operatorname{End}_D(M)\rightarrow E_1(b)$ be a map defined by $f\mapsto f(m)$.
		It is sufficient to show that $\psi$ is surjective.
		One can check that $\psi$ is a linear map.
		Let $g\in\ker(\psi)$. Then $g(m)=0$. Since $M$ is a simple $D$-module and $m\ne 0$, 
		we have that $g(M)=g(\mathbb{F}_pDm)=0$. This implies that $\ker(\psi)=0$.
		Therefore, $\psi$ is injective.
		By {\rm (iii)} and Lemma~\ref{lem:irreducible-representation-dihedral} {\rm (iii)}, 
		$\dim_{\mathbb{F}_p}(\operatorname{End}_D(M))=d/2=\dim(E_1(b))$.
		Hence $\psi$ is surjective. This completes the proof.
	\end{proof}
	\begin{proof}[Proof of Lemma~\ref{lem:faith-rep}]
		Let $\psi$ be a faithful irreducible representation of $C$ and let \[\mathbb{F}_p[x]/(f(x))\] be the field corresponding to $\psi$.
		Then $f(x)$ is an irreducible polynomial of degree $d$ over $\mathbb{F}_p$ and $f(x)\mid x^r-1$.
		Let $\bar{\mathbb{F}}_p$ be the split field of $x^r-1$. 
		Then there exists $\alpha_1,\ldots,\alpha_d\in\bar{\mathbb{F}}_p$
		such that $f(x)=(x-\alpha_1)\cdots(x-\alpha_d)$ and $\alpha_i^r=1$ for every $i\in\{1,\ldots,d\}$.
		If $f(x)\mid x^i-1$ for some $i<r$, then $x$ has order $i$ in $\mathbb{F}_p[x]/(f(x))$.
		However, note that $x$ is of order $r$ in $\mathbb{F}_p[x]/(f(x))$.
		Therefore, $f(x)\nmid x^i-1$ for every $i<r$, which implies that $\alpha_1,\ldots,\alpha_d$ are primitive r-th roots.
		Let $\sigma\in\operatorname{Gal}(\bar{\mathbb{F}}_p/\mathbb{F}_p)$ be the Frobenius map. Then $\sigma(f(x))=f(x)$, which implies that $\{\alpha_1,\ldots,\alpha_d\}$ is invariant under $\langle\sigma\rangle$.
		Since $f(x)$ is irreducible, $\langle \sigma\rangle$ is transitive and fixed-point-free on $\{\alpha_1,\ldots,\alpha_d\}$. 
		Therefore, $\alpha_1^{p^d}=\alpha_1$ which implies that $d$ is the least positive integer such that $r\mid p^d-1$.
		Since there are $\varphi(r)$ primitive $r$-th roots, we have that there are exactly $\varphi(r)/d$ orbits under the action of $\langle \sigma\rangle$.
		Therefore, there are $\varphi(r)/d$ such irreducible polynomial, which implies that there are $\varphi(r)/d$ different faithful irreducible representations of $C$.
		
		Let $\iota:C\rightarrow C$ be a map given by $c\mapsto c^b$.
		Let $k_0$ be the constant term of $f(x)$.
		Then $\psi\circ\iota$ is also a faithful irreducible representation of $C$ and corresponds to the field $\mathbb{F}_p[x]/(k_0^{-1}x^df(x^{-1}))$.
		Note that $k_0^{-1}x^df(x^{-1})=(x-\alpha_1^{-1})\cdots(x-\alpha_d^{-1})$.
		Then $f(x)=k_0^{-1}x^df(x^{-1})$ if and only if $\alpha_1^{p^i}=\alpha_1^{-1}$ for some $i$.
		Therefore, 
		if $d$ is odd or $d$ is even with $p^{d/2}\not\equiv -1\pmod{r}$, then $f(x)\ne x^df(x^{-1})$ which implies that $\psi\not\cong \psi\circ\iota$;
		if $d$ is even and $p^{d/2}\equiv -1\pmod{r}$, then $f(x)\mid x^df(x^{-1})$ which implies that $\psi\cong \psi\circ\iota$.
		Then by \cite[Theorem 3.1]{sim1994Faithful}, we complete the proof.
	\end{proof}
	\bibliography{bib}
	\bibliographystyle{plain}
\end{document}